\documentclass{article}

\usepackage{amsmath, amsthm, amsfonts}
\usepackage{amssymb}
\usepackage{latexsym}
\usepackage{verbatim}
\usepackage{url}

\usepackage[all,cmtip]{xy}
\usepackage[usenames,dvipsnames]{xcolor}
\usepackage{hyperref}
\hypersetup{colorlinks=true,urlcolor=RawSienna,citecolor=RoyalPurple,linkcolor=MidnightBlue}

\usepackage{enumerate}
\usepackage{amssymb}
\usepackage{array}
\usepackage{nicefrac}
\usepackage[section]{placeins}
\usepackage{nicefrac}

\DeclareMathAlphabet{\mathfr}{U}{euf}{m}{n}

\newtheorem{theorem}{Theorem}[section]
\newtheorem{definition}[theorem]{Definition}
\newtheorem{conjecture}[theorem]{Conjecture}
\newtheorem{proposition}[theorem]{Proposition}
\newtheorem{corollary}[theorem]{Corollary}

\newtheorem{lemma}[theorem]{Lemma}
\newtheorem{example}[theorem]{Example}
\newtheorem{remark}[theorem]{Remark}

\newcommand{\norm}[1]{\left\Vert#1\right\Vert}

\newcommand{\Q}{\mathbb Q}
\newcommand{\Qbar}{{\overline{\mathbb Q}}}
\newcommand{\Gal}{\mathrm{Gal}}

\newcommand{\HH}{\mathbb H}
\newcommand{\R}{\mathbb R}

\newcommand{\F}{\mathbb F}

\newcommand{\C}{\mathbb C}

\newcommand{\GL}{\mathrm{GL}}
\newcommand{\M}{\mathrm{M }}

\newcommand{\ab}{\operatorname{ab}}

\newcommand{\End}{\operatorname{End}}

\newcommand{\Hom}{\operatorname{Hom}}

\newcommand{\Frob}{\operatorname{Frob}}
\newcommand{\Fr}{\operatorname{Fr}}
\newcommand{\id}{\operatorname{id}}
\newcommand{\Rec}{\operatorname{Rec}}
\newcommand{\Aut}{\operatorname{Aut}}

\newcommand{\Jac}{\operatorname{Jac}}

\newcommand{\pp}{\mathfrak P}

\newcommand{\p}{\mathfrak{p}}
\newcommand{\s}{\mathfrak{s}}
\newcommand{\Res}{\operatorname{Res}}

\newcommand{\Tw}{\operatorname{Tw}}

\newcommand{\Tr}{\operatorname{Tr}}

\newcommand{\Ker}{\operatorname{Ker}}

\newcommand{\GSp}{\mathrm{GSp}}

\newcommand{\Sp}{\mathrm{Sp}}
\newcommand{\SU}{\mathrm{SU}}
\newcommand{\USp}{\mathrm{USp}}

\newcommand{\ST}{\mathrm{ST}}

\newcommand{\Lef}{\operatorname{L}}

\newcommand{\AST}{\operatorname{AST}}
\newcommand{\topo}{\operatorname{top}}
\newcommand{\WP}{\operatorname{WP}}

\numberwithin{equation}{section}

\usepackage{url}

\newcommand{\Unitary}{\mathrm{U}}
\DeclareMathOperator{\Trace}{Trace}

\newcommand{\Aseq}[1]{\htmladdnormallink{A#1} {http://oeis.org/A#1}}

\newcommand{\cyc}[1]{{\mathrm{C}_#1}}
\newcommand{\dih}[1]{{\mathrm{D}_#1}}
\newcommand{\alt}[1]{{\mathrm{A}_#1}}
\newcommand{\sym}[1]{{\mathrm{S}_#1}}
\newcommand{\symtilde}[1]{{\tilde{\mathrm{S}}_#1}}

\newcommand{\bF}{{\mathbf{F}}}
\newcommand{\bigslant}[2]{{\left.\raisebox{.2em}{$#1$}\middle/\raisebox{-.2em}{$#2$}\right.}}

\renewcommand{\id}{1}

\begin{document}
\title{Sato-Tate distributions of twists of\\$y^2=x^5-x$ and $y^2=x^6+1$}
\author{Francesc Fit\'e and Andrew V. Sutherland}
\date{\today}

\maketitle

\begin{abstract}
We determine the limiting distribution of the normalized Euler factors of an abelian surface $A$ defined over a number field $k$ when $A$ is $\Qbar$-isogenous to the square of an elliptic curve defined over $k$ with complex multiplication.
As an application, we prove the Sato-Tate Conjecture for Jacobians of $\Q$-twists of the curves $y^2=x^5-x$ and $y^2=x^6+1$, which give rise to 18 of the 34 possibilities for the Sato-Tate group of an abelian surface defined over $\Q$. With twists of these two curves one encounters, in fact,  all of the $18$ possibilities for the Sato-Tate group of an abelian surface that is $\Qbar$-isogenous to the square of an elliptic curve with complex multiplication. Key to these results is the \emph{twisting Sato-Tate group} of a curve, which we introduce in order to study the effect of twisting on the Sato-Tate group of its Jacobian.
\end{abstract}
\tableofcontents

\section{Introduction}\label{section: introduction}

Let $A$ be an abelian variety of dimension $g$, defined over a number field $k$.
The generalized \emph{Sato-Tate conjecture}  predicts that the Haar measure of a certain compact subgroup $G$ of the unitary symplectic group $\USp(2g)$ governs the distribution of the normalized Euler factors $\bar{L}_\p(A,T)$, as~$\p$ varies over the primes of  $k$ where $A$ has good reduction. The normalized Euler factor at a prime $\p$   is the polynomial $\bar{L}_\p(A,T)=L_\p(A,T/q^{\nicefrac{1}{2}})$, where $q=\norm{\p}$ is the norm of $\p$, and  $L_\p(A,T)=\prod_{i=1}^{2g}(1-\alpha_iT)$ is the \emph{$L$-polynomial} of $A$ at $\p$.
The polynomial $L_\p(A,T)$ has the defining property that for each positive integer $n$
$$
\#A(\F_{q^n}) = \prod_{i=1}^{2g}(1-\alpha_i^n).
$$

To make this precise, we need to specify the group $G$, and to define what it means for $G$ to ``govern" the distribution of the polynomials $L_\p(A,T)$.
Associated to the abelian variety $A$, Serre \cite{Se12} has defined, in terms of $\ell$-adic monodromy groups, a compact real Lie subgroup of $\USp(2g)$, denoted $\ST(A)$ and called the \emph{Sato-Tate group} of $A$, satisfying the following property: for each prime $\p$ at which $A$ has good reduction, there exists a conjugacy class of $\ST(A)$ whose characteristic polynomial equals $\overline L_\p(A,T):=\sum_{i=0}^{2g}a_i(A)(\p)T^{i}$.\footnote{See also \cite[\S 2]{FKRS} for a brief summary of this construction; there the Sato-Tate group of~$A$ is denoted $\ST_A$, rather than $\ST(A)$.}
For $i=0,1,\ldots,2g$, let $I_i$ denote the interval $\left[-\binom{2g}{i},\binom{2g}{i}\right]$, and consider the map
\begin{equation}\label{equation: projection map}
\Phi_i\colon\ST(A)\subseteq \USp(2g)\rightarrow I_i\subseteq \R
\end{equation}
that sends an element of $\ST(A)$ to the $i$th coefficient of its characteristic polynomial. Let $\mu(\ST(A))$ denote the Haar measure of $\ST(A)$ and let $\Phi_{i,*}(\mu(\ST(A)))$ denote its image on $I_i$ by $\Phi_i$. We can now state the generalized Sato-Tate Conjecture.
\begin{conjecture}\label{conjecture: generalized Sato-Tate} For $i=0,1,\dots, 2g$, the $a_i(A)(\p)$'s are equidistributed on $I_i$ with respect to $\Phi_{i,*}(\mu(\ST(A)))$.\footnote{When we make equidistribution statements, we sort primes in increasing order by norm.} 
\end{conjecture}

The original Sato-Tate conjecture addresses the case where $A$ is an elliptic curve $E/\Q$ without complex multiplication (CM), in which case $g=1$ and $\ST(A)=\USp(2)=\SU(2)$. This case of the conjecture has recently been proved; see \cite[p.~105]{Se12} for a complete list of references. For elliptic curves $E/k$ with complex multiplication,  there are two cases, depending on whether the CM field $M$ is contained in $k$ or not.  In the former case $\ST(E)$ is isomorphic to the unitary group $\Unitary(1)$ (embedded in $\SU(2)$), and in the latter case $\ST(E)$ is isomorphic to the normalizer of $\Unitary(1)$ in $\SU(2)$. Both cases follow from classical results that we recall in  \S\ref{section: Frob conj classes equidist}.

In all three cases arising for $g=1$, it is easy to see that the Sato-Tate group of $E$ is invariant under twisting: if $E'$ is isomorphic to $E$ over $\Qbar$ then $\ST(E')$ is isomorphic to $\ST(E)$.  However, when $g >1$ this is no longer true.

The purpose of this article is to study the possibilities for the Sato-Tate group of the Jacobians of twists of genus 2 curves defined over $\Q$ with many automorphisms (these arise for curves whose Jacobians are $\Qbar$-isogenous to the square of an elliptic curve with complex multiplication), and to prove that in these cases  Conjecture \ref{conjecture: generalized Sato-Tate} is true\footnote{Using the techniques of this article one can obtain analogous results for genus $3$ curves with many automorphisms, such as the Fermat and Klein quartics.}.

The curves we consider give rise to 18 of the 34 Sato-Tate groups that can occur for an abelian surface defined over $\Q$, yet they all lie in one of the two $\Qbar$-isomorphism classes corresponding to the  curves listed in the title of this article. This makes apparent the importance of understanding  the effect of twisting on the Sato-Tate group.

In the remainder of this section, we describe the two points in the moduli space of genus 2 curves that are the object of our study, and state our main result (Theorem \ref{theorem: main result}).  We also describe the numerical computations used to obtain  explicit examples that realize all the possibilities permitted by our main theorem.

Let us first fix some notation.
Throughout this paper $\Qbar$ denotes a fixed algebraic closure of $\Q$ that is assumed to include the number field $k$ and all of its algebraic extensions. 
Let $G_k=\Gal(\Qbar/k)$ denote the absolute Galois group of $k$.
For any algebraic variety $X$ defined over $k$ and any extension $L/k$, we use $X_L$ to denote the algebraic variety defined over $L$ obtained from $X$ by the base change $k\hookrightarrow L$.
For abelian varieties $A$ and $B$ defined over $k$, we write $A\sim B$ to indicate that there is an isogeny between $A$ and $B$ that is defined over $k$.
We may write $A\sim_k B$ to emphasize the field of definition, but this is redundant (to indicate an isogeny defined over an extension $L/k$ we write $A_L\sim B_L$).

\subsection{Two isolated points in the moduli space of genus 2 curves}\label{section: 2 moduli points}

Let $C$ be a curve of genus $g\leq 3$ defined over $k$. In \S \ref{section: twisting group}, we define the \emph{twisting Sato-Tate group $\ST_{\Tw}(C)$ of $C$}, a compact Lie group with the property that the Sato-Tate group of the Jacobian of any twist of $C$ is isomorphic to a subgroup of $\ST_{\Tw}(C)$. There is a well-known bijection between the set of twists of $C$ up to $k$-isomorphism and the cohomology group $H^1(G_k,\Aut(C_\Qbar))$, given by associating to a twist $C'$ of $C$ the class of the cocycle $\xi(\tau):=\phi(^{\tau}\phi)^{-1}$. Here $\phi$ is an isomorphism from $C'_\Qbar$ to $C_\Qbar$. Thus the group $\Aut(C_\Qbar)$ turns out to be a good measure of how complicated the twists of $C$ can be.

For the rest of \S 1 we let $k=\Q$ and $g=2$.
The automorphism group $\Aut(C_\Qbar)$ is then one of the following seven groups:
$$
\cyc 2,\medspace \dih 2,\medspace \dih 4,\medspace \dih 6,\medspace \mathrm C_{10},\medspace 2\dih 6,\medspace \symtilde 4.
$$
Here $\cyc n$ denotes the cyclic group of $n$ elements, $\dih n$ denotes the dihedral group of order $2n$, and $\sym n$ is the symmetric group on $n$ letters.
The groups $2\dih 6$ and $\symtilde 4$ denote certain double covers of $\dih 6$ and $\sym 4$ that are defined in section \S \ref{section: curves}. In the generic case, $\Aut(C_\Qbar)$ is isomorphic to $\cyc 2$.  This implies that every twist $C'$ of $C$ is quadratic, and we have $\ST(\Jac(C'))=\ST(\Jac(C))=\ST_{\Tw}(C)$.

We are interested in the opposite situation: the two exotic cases in which $\Aut(C_\Qbar)$ is as large as possible: $2\dih 6$ and $\symtilde 4$.   All genus 2 curves $C$ with $\Aut(C_\Qbar)$ isomorphic to $2\dih 6$ (resp.\ $\symtilde 4$) are isomorphic to
\begin{equation}\label{equation: curves}
y^2=x^6+1\qquad \text{(resp.\ $y^2=x^5-x$)},
\end{equation}
thus they constitute a single $\Qbar$-isomorphism class $\mathcal C_3$ (resp.\ $\mathcal C_2$) of curves, an isolated point in the moduli space of all genus 2 curves.

We shall choose representative curves $C^0_2$ and $C^0_3$ for $\mathcal C_2$ and $\mathcal C_3$ that are defined over $\Q$ and have particularly nice arithmetic properties.  We write $C^0$ (resp.\ $\mathcal C$) to denote either $C^0_2$ or $C^0_3$ (resp., either $\mathcal C_2$ or $\mathcal C_3$).
The key arithmetic property we require of $C^0$ is that its Jacobian be $\Q$-isogenous to $E^2$, where $E$ is an elliptic curve defined over $\Q$ (with CM).
This applies only to the curve $y^2=x^6+1$ listed in (\ref{equation: curves}), which we take as our representative $C^0_3$ for the class~$\mathcal C_3$, but it also applies to the curve
\begin{equation}
y^2=x^6-5x^4-5x^2+1,
\end{equation}
which we take as a \emph{better} representative $C^0_2$ for the class $\mathcal C_2$ of $y^2=x^5-x$.

The classification in \cite{FKRS} gives an explicit description of each of the 52 Sato-Tate groups that can and do arise in genus 2, as subgroups of $\USp(4)$, of which 32 have identity component (isomorphic to) $\Unitary(1)$.
The two curves listed in (\ref{equation: curves}) both appear in \cite{FKRS}, where they are shown to have Sato-Tate groups with identity component $\Unitary(1)$.
It follows that if $C$ is a twist of either of these curves, then $\ST(\Jac(C))$ also has identity component $\Unitary(1)$.
In fact, the representative curves for all 32 of the $\Unitary(1)$ cases listed in \cite{FKRS} are actually twists of one of the two curves in (\ref{equation: curves}) (possibly using an extended field of definition).

Among the 32 genus 2 Sato-Tate groups with identity component $\Unitary(1)$, two are maximal.  The first has component group $\sym 4 \times \cyc 2$ and is denoted $J(O)$, while the second has component group $\dih 6 \times \cyc 2$ and is denoted $J(D_6)$.
We will prove that $\ST_{\Tw}(C^0_2)=J(O)$ and $\ST_{\Tw}(C^0_3)=J(D_6)$, and, as a consequence, that the Sato-Tate group of any twist of $C^0_2$ (resp.\ $C^0_3$) is isomorphic to a subgroup of $J(O)$ (resp.\ $J(D_6)$).
Conversely, we will show that every Sato-Tate group that can occur over $\Q$ and is isomorphic to a subgroup of $J(O)$ (resp.\ $J(D_6)$) arises for some $\Q$-twist $C$ of $C^0_2$ (resp.\ $C^0_3$), by giving explicit examples in each case.\footnote{We call $C$ a $\Q$-\emph{twist} of $C^0$ if $C$ is defined over $\Q$ and $C_\Qbar\simeq C^0_\Qbar$.}
Most of the Sato-Tate groups $G$ with identity component $\Unitary(1)$ are actually subgroups of \emph{both} $J(O)$ and $J(D_6)$.
In such cases we exhibit $\Q$-twists of both $C^0_2$ and $C^0_3$ that have Sato-Tate group $G$.

\subsection{Main result}\label{subsection: main result}

Recall that $C^0$ denotes either $C^0_2$ or $C^0_3$. 
These are both genus 2 curves defined over $\Q$ whose Jacobians are $\Q$-isogenous to the square of an elliptic curve $E/\Q$ with CM by an imaginary quadratic field $M$ equal to $\Q(\sqrt{-2})$ or $\Q(\sqrt{-3})$, respectively.
Our main result is that Conjecture \ref{conjecture: generalized Sato-Tate} holds for the Jacobians of the $\Q$-twists $C$ of $C^0$.

In order to state the theorem more precisely, we introduce some notation.

\begin{definition}\label{definition: triple} For any $\Q$-twist $C$ of $C^0$, let $K/\Q$ (resp.\ $L/\Q$) denote the minimal extension over which all endomorphisms of $\Jac(C)_\Qbar$ (resp.\ homomorphisms from $\Jac(C)_\Qbar$ to $E_\Qbar$) are defined. Then we write $T(C)$ for the isomorphism class
$$[\Gal(L/\Q),\Gal(K/\Q),\Gal(L/M)]\,.$$
\end{definition}
We say that two triples of groups $(H_1,H_2,H_3)$ and $(H_1',H_2',H_3')$ are isomorphic if $H_i\simeq H_i'$ for $i=1,2,3$. We write $[H_1,H_2,H_3]$ for the isomorphism class of $(H_1,H_2,H_3)$, which we regard as a triple of abstract groups.

\begin{definition}\label{definition: zvector}
For any finite group $H$ with a subgroup $H_0$ and a normal subgroup~$N$, let $o(s,r)$ (resp.\ $\overline o(s,r)$) count the elements in $H_0$ (resp.\ $H\setminus H_0$) of order $s$ whose projection in $H/N$ has order $r$.
Let $z(H,N,H_0)$ denote the vector $[z_1,z_2]$, where
\begin{align*}
z_1&=[o(1,1),o(2,1),o(2,2),o(3,3),o(4,2),o(6,3),o(6,6),o(8,4),o(12,6)]\,,\\
z_2&=[\bar{o}(2,2),\bar{o}(4,2),\bar{o}(6,6),\bar{o}(8,4),\bar{o}(12,6)]\,.
\end{align*}
For any $\Q$-twist $C$ of $C^0$, write $$z(C):=[z_1(C),z_2(C)]:=z(\Gal(L/\Q),\Gal(L/K),\Gal(L/M))\,.$$
\end{definition}
We also define $o(r)=\sum_s o(s,r)$ and $\bar o(s)=\sum_r \bar o(s,r)$.
We note that in the cases of interest, $z(H,N,H_0)$ is $z(C)$ for some $\Q$-twist $C$ of $C^0$.
In this situation, $o(r)$ is the number of elements in $\Gal(L/M)$ whose projection to $\Gal(K/M)$ has order $r$, and $\overline o(s)$ is the number of elements of order $s$ in $\Gal(L/\Q)$ that are not in $\Gal(L/M)$.  Clearly
$$
\sum o(s,r)=\sum \overline o(s,r)=|\Gal(L/\Q)|/2.
$$
Moreover, we prove in Proposition \ref{proposition: trthetaMEA} that the only pairs $(r,s)$ for which $o(r,s)$ or $\bar{o}(r,s)$ can be nonzero are those that appear in the vectors $z_1$ and $z_2$.

Finally, let $L_p(C,T)$ denote the Euler factor of $C$ at a prime $p$ of good reduction.
We may write the normalized Euler factor $\overline L_{p}(C,T)=L_p(C,T/p^{\nicefrac{1}{2}})$ as
$$
\overline L_{p}(C,T)=T^4 + a_1(C)(p)T^3 + a_2(C)(p)T^2+ a_1(C)(p)T + 1.
$$
We are now ready to state our main theorem.

\begin{theorem}\label{theorem: main result} Let $C$ be a $\Q$-twist of $C^0$.
\begin{enumerate}[\rm (i)]
\item There are exactly $20$ possibilities for $T(C)$ if $C^0=C^0_2$, and $21$ if $C^0=C^0_3$.
\item The triple $T(C)$ and the vector $z(C)$ uniquely determine each other.
\item The triple $T(C)$ (or $z(C)$) determines the Sato-Tate group $\ST(\Jac(C))$.
\item  For $i=1,2$, the $a_i(C)(p)$'s are equidistributed on $I_i=[-\binom{4}{i},\binom{4}{i}]$ with respect to a measure $\mu(a_i(C))$ that is uniquely determined by  the vector $z(C)$. More precisely, the density function of $\mu(a_i(C))$ is continuous up to a finite number of points, and it is therefore uniquely determined by its moments:
$$
\begin{array}{ll}
\M_n[\mu(a_1(C))]=&\frac{1}{[L:\Q]}\bigl(o(1)2^n+o(3) + o(4)2^{n/2} + o(6)3^{n/2}\bigr) b_{0,n}\,,\\[6pt]
\M_n[\mu(a_2(C))]=&\frac{1}{[L:\Q]}\bigl(o(1)b_{4,n} + o(2)b_{0,n} + o(3)b_{1,n} + o(4)b_{2,n} + o(6)b_{3,n} \\[6pt]
 & \qquad\quad +\  \bar o(2)2^n + \bar o(4)(-2)^n + \bar o(6)(-1)^n + \bar o(12))\,.
\end{array}
$$
Here $b_{m,n}$ denotes the coefficient\footnote{For $m=0,1,2,3,4$ the $b_{m,n}$ form the sequences
\Aseq{126869}, \Aseq{0002426}, \Aseq{000984}, \Aseq{026375}, \Aseq{081671}, respectively, in the Online Encyclopedia of Integer Sequences \cite{OEIS}.} of $X^n$ in $(X^2+mX+1)^n$.
\item Conjecture \ref{conjecture: generalized Sato-Tate} holds for $C$.
\end{enumerate}
\end{theorem}

We actually prove statement (iv) in greater generality, for an abelian surface~$A$ defined over a number field $k$ with $A_\Qbar\sim E^2_\Qbar$, where $E$ is an elliptic curve defined over $k$ with CM by a quadratic imaginary field $M$.
This is accomplished in \S\ref{section: squares} via Corollary \ref{corollary: a2 moments}, whose proof relies on a study of the structure of $\Hom(E_L,A_L)\otimes_M\Qbar$ as a Galois $\Qbar[\Gal(L/M)]$-module and a refined equidistribution statement of Frobenius elements of a CM elliptic curve when restricted to certain Galois conjugacy classes (see Corollary \ref{corollary: equidist frobenius classes}).
We compute the moments
$$
\M_n[a_i(C)]:=\lim_{x\to\infty}\frac{1}{\pi(x)} \sum_{p\le x}a_i(C)(p)^n,
$$
where $p$ varies over primes of good reduction, and prove equidistribution of the $a_i(C)(p)$'s with respect to a measure $\mu(a_i(C))$.  It follows that $\M_n[\mu(a_1(C))]=\M_n[a_i(C)]$. We devote $\S$\ref{section: curves} to the proofs of assertions (i), (ii) and (iii), which follow from Corollary \ref{corollary: num g(C)}, Proposition \ref{proposition: candidates}, and Proposition \ref{proposition: det ST}, respectively.
The final assertion (v) follows from (iii) and (iv): it is enough to check that  
for each of the 41 possibilities of $T(C)$, the formulas obtained for $\mu[a_i(C)]$ coincide with the ones obtained for $\Phi_{i,*}(\mu(\ST(\Jac(C))))$ in \cite{FKRS}.

\subsection{Numerical computations}

In $\S$\ref{section: computations} we show that all 41 of the possible triples $T(C)$ determined in section $\S$\ref{section: curves} actually arise for some $\Q$-twist $C$ of $C^0$ by exhibiting a provable example of each case.
The example curves $C$ were obtained by an extensive search that was made feasible by part (ii) of Theorem \ref{theorem: main result};
it is computationally much easier to approximate $z(C)$ than it is to explicitly compute $T(C)$, which requires computing the Galois groups of number fields of fairly large degree (48 or 96 in the most typical cases).

For an elliptic curve $E$ with CM, the values $a_1(E)(p)$ can be computed very quickly, and we show how to compute $a_1(C)(p)$ and $a_2(C)(p)$ from $a_1(E)(p)$, using the fact that $\Jac(C)$ is $\Qbar$-isogenous to $E^2$ (see Proposition \ref{proposition: trthetaMEA}).
This allows us to efficiently compute an approximation of $z(C)$ (using again Proposition \ref{proposition: trthetaMEA}) of precision sufficient to provisional identify $T(C)$ (via part (ii) of Theorem \ref{theorem: main result}).
Many curves were analyzed (tens of thousands) in order to obtain~41 candidate examples, one for each possible triple $T(C)$.
For each of these~41 candidates we then proved that the provisional identification of $T(C)$ is correct, by explicitly computing the Galois groups $\Gal(L/\Q)$, $\Gal(K/\Q)$, and $\Gal(L/M)$.

\subsection{Acknowledgements}
We thank Joan-C. Lario and Jordi Quer for helpful comments, and we are grateful to Kiran S. Kedlaya for indicating the way to prove Proposition \ref{proposition: equidistribution classes}.
Fit\'e thanks the Massachusetts Institute of Technology for its warm hospitality from September to December 2011, the period in which this project was realized. Fit\'e is also grateful to the University of Cambridge for its welcome in the period May-July 2011, when some of the ideas of this article originated.
Fit\'e  received financial support from DGICYT grant MTM2009-13060-C02-01 and from NSF grant DMS-0545904.
Sutherland received financial support from NSF
grant DMS-1115455.

\section{The twisting Sato-Tate group of a curve}\label{section: twisting group}

In this section we define the \emph{twisting Sato-Tate group}, which is our main object of study.
We do so in terms of the \emph{algebraic Sato-Tate group} defined by Banaszak and Kedlaya in \cite{BK}.
Let $A$ be an abelian variety of dimension $g\leq 3$ defined over a number field~$k$, and fix an embedding of $k$ into $\C$ .
Fix a polarization on~$A$ and a symplectic basis for the singular homology group $H_1(A_\C^{\topo},\Q)$.
Use it to equip this space with an action of $\GSp_{2g}(\Q)$. For each $\tau \in G_k$, define
\begin{equation}\label{equation: Lef}
\Lef(A,\tau) := \{\gamma \in \Sp_{2g}: \gamma^{-1} \alpha \gamma = {}^\tau\alpha \mbox{ for all $\alpha \in \End(A_\Qbar)\otimes{\Q}$}\}.
\end{equation}
Here we view $\alpha$ as an endomorphism of $H_1(A_\C^{\topo}, \Q)$.
The \emph{algebraic Sato-Tate group of $A$} is defined by
$$
\AST(A):=\bigcup_{\tau\in G_k} \Lef(A,\tau).
$$
The Sato-Tate group $\ST(A)$ is a maximal compact subgroup of $\AST(A)\otimes_\Q\C$; see \cite[Thm.~6.1, Thm.~6.10]{BK}.

\begin{remark}\label{remark: twist invariance}
As noted in the introduction, $\ST(A)$ is invariant under twisting when $g=1$.
This does not hold for $g>1$, however $\ST(A)$ \emph{is} invariant under quadratic twisting.
For $g\le 3$ this follows easily from the definitions above.
Indeed, let $\chi\colon G_k\rightarrow\C$ be a quadratic character. For every $\tau\in G_k$, one has $\Lef(A\otimes\chi,\tau)=\Lef(A,\tau)\otimes\chi(\tau)$ (see (\ref{equation: Twisted Lef}) for a more general relation).
Invariance under quadratic twisting follows from the fact that $\Lef(A,\tau)\otimes\chi(\tau) = \Lef(A,\tau)$.
For $A$ of arbitrary dimension, the invariance of $\ST(A)$ under quadratic twisting follows easily from the definition of $\ST(A)$ given in \cite{Se12} (see also \cite{FKRS}), in terms of the image of the $\ell$-adic representation attached to $A$.
\end{remark}

We now assume that $A$ is the Jacobian $\Jac(C)$ of a curve $C$ defined over $k$, and view $\Aut(C_\Qbar)$ as a subgroup of $\GL(H_1(\Jac(C)_{\C}^{\topo}, \Q))$.

\begin{definition} The \emph{twisting algebraic Sato-Tate group of $C$} is the algebraic subgroup of $\Sp_{2g}/\Q$ defined by
$$
\AST_{\Tw}(C):=\AST(\Jac(C))\cdot\Aut(C_\Qbar) \,.
$$
\end{definition}
Observe that $\AST_{\Tw}(C)$ is indeed a group: for any $\gamma_1,\gamma_2\in\AST(\Jac(C))$ and $\alpha_1,\alpha_2\in\Aut(C_\Qbar)$, we have
$$\gamma_1\alpha_1(\gamma_2\alpha_2)^{-1}=\gamma_1\gamma_2^{-1}\gamma_2[\alpha_1\alpha_2^{-1}]\gamma_2^{-1}=\gamma_1(\gamma_2^{-1}){}^{\tau_2^{-1}}(\alpha_1\alpha_2)\in \AST_{\Tw}(C)\,.$$

Now let $C'$ be a \emph{twist of $C$}, a curve defined over $k$ for which $C'_L\simeq C_L$ for some finite Galois extension $L/k$.
Let $\phi\colon C'_L\rightarrow C_L$ be a fixed isomorphism. It is easy to check that
\begin{equation}\label{equation: Twisted Lef}
\Lef(\Jac(C'),\tau)=\phi^{-1}\Lef(\Jac(C),\tau)({}^{\tau}\phi)\,.
\end{equation}
Here $\phi$ is seen as a homomorphism from $H_1(\Jac(C')_\C^{\topo}, \Q)$ to $H_1(\Jac(C)_\C^{\topo}, \Q)$.

\begin{lemma}\label{lemma: inclusion AST} Let $\gamma'\in\Lef(\Jac(C'),\tau)\subseteq \AST(\Jac(C'))$.
Write $\gamma'$ as $\phi^{-1}\gamma({}^{\tau}\phi)$ with $\gamma$ in $\Lef(\Jac(C),\tau)$, as in (\ref{equation: Twisted Lef}). The map
$$\Lambda_\phi\colon \AST(\Jac(C'))\rightarrow \AST_{\Tw}(C)\,,\qquad \Lambda_\phi(\gamma')=\gamma({}^\tau\phi)\phi^{-1}$$
is a (well-defined) monomorphism of groups.
\end{lemma}

\begin{proof} Let $\gamma_1'=\phi^{-1}\gamma_1({}^{\tau_1}\phi)$ and $\gamma_2'=\phi^{-1}\gamma_2({}^{\tau_2}\phi)$ be elements of $\Lef(\Jac(C'),\tau_1)$ and $\Lef(\Jac(C'),\tau_2)$, respectively. Then

$$
\begin{array}{l@{\,=\,}l}
\Lambda_\phi(\gamma_1'\gamma_2') & \displaystyle{\Lambda_\phi(\phi^{-1}\gamma_1\gamma_2\gamma_2^{-1}[({}^{\tau_1}\phi)\phi^{-1}]\gamma_2({}^{\tau_2}\phi))}\\[6pt]
 & \displaystyle{\Lambda_\phi(\phi^{-1}\gamma_1\gamma_2({}^{\tau_2\tau_1}\phi)({}^{\tau_2}\phi)^{-1}({}^{\tau_2}\phi))}\\[6pt]
 & \displaystyle{\Lambda_\phi(\phi^{-1}\gamma_1\gamma_2({}^{\tau_2\tau_1}\phi))=\gamma_1\gamma_2({}^{\tau_2\tau_1}\phi)\phi^{-1}}\\[6pt]
 & \displaystyle{\gamma_1\gamma_2[({}^{\tau_2\tau_1}\phi)({}^{\tau_2}\phi)^{-1}]\gamma_2^{-1}\gamma_2({}^{\tau_2}\phi)\phi^{-1}}\\[6pt]
 & \displaystyle{\gamma_1({}^{\tau_1}\phi)\phi^{-1}\gamma_2({}^{\tau_2}\phi)\phi^{-1}=\Lambda_\phi(\gamma_1')\Lambda_\phi(\gamma_2')\,.}\\[6pt]
\end{array}
$$
It is clear that $\Lambda_\phi$ is both well defined and injective: $\Lambda_\phi(\gamma'_1)=\Lambda_\phi(\gamma'_2)$ if and only if $\gamma_1'=\gamma_2'$.
\end{proof}

We now define the \emph{twisting Sato-Tate group} $\ST_{\Tw}(C)$ of $C$.

\begin{definition} The \emph{twisting Sato-Tate group $\ST_{\Tw}(C)$ of $C$} is a maximal compact subgroup of $\AST_{\Tw}(C)\otimes \C$.
\end{definition}

\begin{remark} It follows from the previous lemma that for any twist $C'$ of $C$, the Sato-Tate group $\ST(\Jac(C))$ is isomorphic to a subgroup of $\ST_{\Tw}(C)$. We also note that the component groups of $\ST_{\Tw}(C)$ and $\AST_{\Tw}(C)\otimes \C$ must be isomorphic, and the identify components of $\ST_{\Tw}(C)$ and $\ST(\Jac(C))$ are equal.
\end{remark}

Our next goal is to study the component group of $\ST_{\Tw}(C)$ when $C$ is a hyperelliptic curve (of genus $g\leq 3$). Consider the group\footnote{The product of elements $(\alpha_1,\gamma_1)$ and $(\alpha_2,\gamma_2)$ in $\Aut(C_\Qbar)\rtimes \AST(\Jac(C))$ is defined to be $(\alpha_2\gamma_2^{-1}\alpha_1\gamma_2,\gamma_1\gamma_2)= (\alpha_2\cdot{}^{\tau_2}\alpha_1,\gamma_1\gamma_2)$, where $\gamma_2\in \Lef(A,\tau_2)\subset\AST(\Jac(C))$.}
$$\bigslant{\Aut(C_\Qbar)\rtimes \AST(\Jac(C))}{Z}\,,$$
where $Z$ is the normal subgroup of $\Aut(C_\Qbar)\rtimes \AST(\Jac(C))$ consisting of the pairs $(\alpha,\gamma)$ with $\alpha=\gamma$, where $\alpha\in\Aut(C_\Qbar)$ and $\gamma\in\AST(\Jac(C))$.

\begin{lemma}\label{lemma: isomorphism Phi} The map
$$
\Phi\colon \AST_{\Tw}(C)\rightarrow\bigslant{\Aut(C_\Qbar)\rtimes \AST(\Jac(C))}{Z}\,,\qquad \Phi(\gamma\alpha)=(\alpha^{-1},\gamma)
$$
is a (well-defined) isomorphism.
\end{lemma}
\begin{proof} For any $\gamma_1,\gamma_2\in\AST(\Jac(C))$ and $\alpha_1,\alpha_2\in\Aut(C_\Qbar)$, we have
$$
\begin{array}{l@{\,=\,}l}
\Phi(\gamma_1\alpha_1\gamma_2\alpha_2) & \displaystyle{\Phi(\gamma_1\gamma_2({}^{\tau_2}\alpha_1)\alpha_2)}=(\alpha_2^{-1}({}^{\tau_2}\alpha_1)^{-1},\gamma_1\gamma_2)\\[6pt]
 & \displaystyle{(\alpha_1^{-1},\gamma_1)(\alpha_2^{-1},\gamma_2)=\Phi(\gamma_1\alpha_1)\Phi(\alpha_2\gamma_2)\,.}\\[6pt]
\end{array}
$$
The surjectivity of $\Phi$ is clear. It remains to prove that $\Phi(\gamma_1\alpha_1)=\Phi(\gamma_2\alpha_2)$ if and only if $\gamma_1\alpha_1=\gamma_2\alpha_2$. On the one hand, $\Phi(\gamma_1\alpha_1)=\Phi(\gamma_2\alpha_2)$ if and only if
$$Z\ni(\alpha_2^{-1},\gamma_2)(\alpha_1^{-1},\gamma_1)^{-1}=({}^{\tau_1^{-1}}(\alpha_1\alpha_2^{-1}),\gamma_2\gamma_1^{-1})\,.$$
On the other hand, $\gamma_1\alpha_1=\gamma_2\alpha_2$ if and only if $\alpha_1\alpha_2^{-1}=\gamma_1^{-1}\gamma_2$, equivalently,  ${}^{\tau_1^{-1}}(\alpha_1\alpha_2^{-1})=\gamma_2\gamma_1^{-1}$.  But then $({}^{\tau_1^{-1}}(\alpha_1\alpha_2^{-1}),\gamma_2\gamma_1^{-1})\in Z\,$.
\end{proof}

We now assume $C$ is a hyperelliptic curve (of genus $g\leq 3$). As an endomorphism of $H_1(\Jac(C)_\C^{\topo}, \Q)$, the hyperelliptic involution $w$ of $C$ corresponds to the matrix $-1\in\Sp_{2g}(\Q)$. Recall that $\AST(\Jac(C))$ contains the matrix $-1$. Thus $(-1,-1)\in Z$ and Lemma \ref{lemma: isomorphism Phi} implies that $\AST_{\Tw}(C)$ is isomorphic to a subgroup of
$$
 \bigslant{\Aut(C_\Qbar)\rtimes \AST_{\Jac(C)}}{\langle (-1,-1)\rangle}\,.
$$

Let $K/k$ denote the minimal field extension over which all the endomorphisms of $\Jac(C)$ are defined. Then, since the component group of $\ST(\Jac(C))$ is isomorphic to $\Gal(K/k)$ (see \cite[Rem.~6.4, Thm.~6.10]{BK}), and the identity component of $\ST(\Jac(C))$ contains the matrix $-1$, the component group of $\ST_{\Tw}(C)$ is isomorphic to a subgroup of
$$
\bigslant{\Aut(C_\Qbar)}{\langle w\rangle}\rtimes \Gal(K/k)\,.
$$

By lemma \ref{lemma: inclusion AST}, for any twist $C'$ of $C$, there exists a monomorphism of groups
\begin{equation}\label{remark: injection}
\overline\lambda_\phi\colon\Gal(K/k)\rightarrow \bigslant{\Aut(C_\Qbar)}{\langle w\rangle}\rtimes \Gal(K/k)\,.
\end{equation}
It follows that if there exists a twist $\tilde C$ of $C$ such that
\begin{equation}\label{equation: orders}
|\Gal(\tilde K/k)|=|\Aut(C_\Qbar)|\cdot|\Gal(K/k)|/2\,,
\end{equation}
where $\tilde K/k$ is the minimal extension over which all the endomorphisms of $\Jac(\tilde C)$ are defined, then $\ST_{\Tw}(C)=\ST(\Jac(\tilde C))$, and for every twist $C'$ of $C$, the Sato-Tate group $\ST(\Jac(C'))$ is a subgroup of $\ST(\Jac(\tilde C))$.

\begin{remark} Let $C^0_2$ and $C^0_3$ be the two curves defined in \S\ref{section: 2 moduli points}. If $\tilde C$ is a twist of $C^0_2$ (resp.\ $C^0_3$) such that $\ST(\tilde C)=J(O)$ (resp.\ $J(D_6)$), then equation (\ref{equation: orders}) is satisfied.
It follows that $\ST_{\Tw}(C^0_2)=J(O)$ and $\ST_{\Tw}(C^0_2)=J(D_6)$.
\end{remark}

\section{Squares of CM elliptic curves}\label{section: squares}

We shall work in the category of abelian varieties up to isogeny, so we call the elements of $\Hom(A,B)\otimes \Q$ homomorphisms, the elements of $\End(A)\otimes\Q$ endomorphisms, and the surjective elements in $\Hom(A,B)\otimes\Q$ isogenies.

We henceforth assume that $A$ is an abelian variety over $k$ such that $A_\Qbar\sim E_\Qbar^2$, where $E$ is an elliptic curve defined over $k$ with complex multiplication (CM) by an imaginary quadratic field $M$ (except in \S\ref{subsection: additional}, where we do not assume $E$ has CM).
Let $L/k$ be the minimal extension over which all the homomorphisms from $E_\Qbar$  to $A_\Qbar$ are defined, and let $K/k$  be the minimal extension over which all the endomorphisms of $A_\Qbar$ are defined.
We note that $kM\subseteq K\subseteq L$, and we have $\Hom(E_\Qbar,A_\Qbar)\simeq\Hom(E_{L} ,A_{L} )$ and $A_{L} \sim E_{L} ^2$.

\subsection{The Galois modules \texorpdfstring{$\Hom(E_L,A_L)$}{} and \texorpdfstring{$\End(A_L)$}{}}\label{section: representations general}\label{section: Galois modules}

Let $\sigma$ and $\overline\sigma$ denote the two embeddings of $M$ into $\Qbar$.
Consider
$$
\Hom(E_{L} ,A_{L} )\otimes_{M,\sigma}\Qbar \qquad \text{(resp.\ $\End(A_{L} )\otimes_{M,\sigma}\Qbar$)}\,,
$$
where the tensor product is taken via the embedding $\sigma\colon M\hookrightarrow\Qbar$. Letting $\Gal(L/kM)$ act trivially on $\Qbar$ and naturally on $\Hom(E_{L} ,A_{L} )$, it acquires the structure of a $\overline \Q[\Gal(L/kM)]$-module of dimension 2 (resp.~4) over~$\Qbar$, and similarly for $\overline\sigma$.

\begin{definition}\label{definition: thetaMsigma}
Let $\theta:=\theta_{M,\sigma}(E,A)$ (resp.\ $\theta_{M,\sigma}(A)$) denote the representation afforded by the module $\Hom(E_{L} ,A_{L} )\otimes_{M,\sigma}\Qbar$ (resp.\ $\End(A_{L} )\otimes_{M,\sigma}\Qbar$), and similarly define $\overline\theta := \theta_{M,\overline\sigma}(E,A)$ and $\theta_{M,\overline\sigma}(A)$.
Let $\theta_\Q:=\theta_\Q(E,A)$ (resp.\ $\theta_\Q(A)$) denote the representation afforded by the $\Q[\Gal({L} /k)]$-module $\Hom(E_{L},A_{L})\otimes\Q$ (resp.\ $\End(A_L)\otimes\Q$).
\end{definition}
For each $\tau \in \Gal({L} /kM)$ we write
$$
\det(1-\theta(\tau)T)=1+a_1(\theta)(\tau)T+a_2(\theta)(\tau)T^2\,,
$$
where $a_1(\theta)=\Tr\theta$ and $a_2(\theta)=\det(\theta)$ are elements of $M$.
Observe that 
\begin{equation}\label{equation: trace}
\Tr\theta_\Q(\tau)=\Tr_{M/\Q}\Tr\theta(\tau)\qquad \text{if $\tau\in\Gal(L/kM)$.}
\end{equation}
For $z\in M$ let $|z|:=\sqrt{\sigma(z)\overline\sigma(z)}$.

\begin{proposition}\label{proposition:norm} There is an isomorphism of $\Qbar[\Gal(L/kM)]$-modules
$$
\End(A_{L} )\otimes_{M,\sigma}\Qbar\simeq \left(\Hom(E_{L} ,A_{L} )\otimes_{M,\sigma}\Qbar\right)^*\otimes \Hom(E_{L} ,A_{L} )\otimes_{M,\sigma}\Qbar\,.
$$
Thus $\Tr\theta_{M,\sigma}(A)=\Tr\theta_{M,\overline\sigma}(E,A)\cdot\Tr\theta_{M,\sigma}(E,A)=|\Tr(\theta)|^2\in\Q$,
and therefore $\theta_{M,\sigma}(A)\simeq\theta_{M,\overline\sigma}(A)$.
\end{proposition}

\begin{proof} Consider the natural inclusion of $\Qbar[\Gal(L/kM)]$-modules
$$
\End(A_{L} )\otimes_{M,\sigma}\Qbar \hookrightarrow \Hom_\Qbar(\Hom(E_{L} ,A_{L} )\otimes_{M,\sigma}\Qbar ,\Hom(E_{L} ,A_{L} )\otimes_{M,\sigma}\Qbar )\,,
$$
which sends an element $\psi$ in $\End(A_{L} )\otimes_{M,\sigma}\Qbar$ to the linear map of $\Qbar$-vector spaces that sends $f$ in $\Hom(E_{L} ,A_{L} )\otimes_{M,\sigma}\Qbar$ to $\psi\circ f$ in $\Hom(E_{L} ,A_{L} )\otimes_{M,\sigma}\Qbar$.
Both spaces have dimension $4$ over $\Qbar$, and therefore must be isomorphic as $\Qbar[\Gal(L/kM)]$-modules.
\end{proof}

Let $\pi\colon \Gal({L} /kM)\rightarrow \Gal(K/kM)$ be the natural projection.
For each $\tau$ in $\Gal({L} /kM)$, let $s=s(\tau)$ denote the order of $\tau$ and let $r=r(\tau)$ denote the order of $\pi(\tau)$ in $\Gal(K/kM)$.
The possible values of $r$ are $1$, $2$, $3$, $4$, and $6$;  see \cite[\S4.5]{FKRS}.

\begin{proposition}\label{proposition:trthetaQ}
Suppose $\tau\in\Gal(L/k)$ does not lie in $\Gal(L/kM)$.  Then the eigenvalues of $\theta_\Q(E,A)(\tau)$ are as follows:
\begin{center}
\begin{tabular}{lllll}
$s=2\colon$ & $-1,-1,1,1$            & $\quad$ & $s=8\colon$ &  $\zeta_8,\zeta_8^3,\zeta_8^5,\zeta_8^7$\vspace{2pt}\\
$s=4\colon$ & $i,i,-i,-i$& $\quad$  & $s=12\colon$ & $\zeta_{12},\zeta_{12}^5,\zeta_{12}^7,\zeta_{12}^{11}$\vspace{2pt}\\
$s=6\colon$ &  $\zeta_3,\zeta_3^2,\zeta_6,\zeta_6^5$
\end{tabular}
\end{center}
Here, $\zeta_r$ stands for an $r$th root of unity.
\end{proposition}

\begin{proof}
We can assume that $kM/k$ is quadratic, otherwise there is nothing to prove. We first show the following properties of $\theta_\Q(E,A)$:
\begin{enumerate}[(\rm i)]
\item The least common multiple of the orders of the eigenvalues of $\theta_\Q(E,A)(\tau)$ is equal to $s$.
\item If $\tau\in\Gal(L/k)\setminus \Gal(L/kM)$ then $\Tr\theta_\Q(E,A)(\tau)=0$.
\end{enumerate}
It follows from the definition of $L/k$ that the representation $\theta_\Q(E,A)$ is faithful, which implies (i).
Let $\chi$ be the quadratic character of $\Gal(L/k)$ associated to the quadratic extension $kM/k$.
Then $E\otimes \chi\sim_k E$ (and, in fact, $A\otimes \chi\sim_k A$), which implies that $\Hom(E_{L},A_{L})=\Hom(E_{L},A_{L})\otimes \chi$ (by \cite[Prop.~1.6]{MRS07}, for example). This proves (ii). 

For $s=2,6,8,12$, the proposition follows from (i) and (ii).
For $s=4$, (i) implies that $i$ is an eigenvalue of $\theta_\Q(E,A)(\tau)$ and (ii) leaves just two possibilities for the four eigenvalues: $i, -i, 1, -1$, or $i, -i, i, -i$.
We now show that only the latter can arise.
The eigenvalues of $\theta_\Q(E,A)(\tau)$ are quotients of roots of $\overline L_{\p}(E,T)$ and roots of $\overline L_{\p}(A,T)$, where $\p$ is a prime of $k$, inert in $kM$, of good reduction for $A$ and $E$. We can further assume that $\p$ has absolute degree 1.
Then $\overline L_\p(E,T)=1+T^2$, and the polynomial $\overline L_{\p}(A,T)$ is one of the following:
\begin{equation}\label{equation: five possibilities}
(1-T^2)^2,\quad 1-T^2+T^4,\quad 1+T^4,\quad 1+T^2+T^4,\quad (1+T^2)^2\,.
\end{equation}
In no case can both 1 and $i$ arise as quotients of a root of $\overline L_\p(E,T)=1+T^2$ and roots of $\overline L_{\p}(A,T)$.
\end{proof}

In view of Proposition~\ref{proposition:norm}, we write $\theta_{M}(A)$ for $\theta_{M,\sigma}(A)\simeq\theta_{M,\overline\sigma}(A)$.

\begin{proposition}\label{proposition:trtheta} For each $\tau\in\Gal(L/kM)$ we have
$$
\Tr\theta_M(A)(\tau)=2+\zeta_{r}+\overline \zeta_{r}\,.
$$
\end{proposition}

\begin{proof}
It follows from \cite[Prop.~9]{FKRS} that the eigenvalues of $\theta_\Q(A)(\tau)$ are $1$, $1$, $1$, $1$, $\zeta_r$, $\zeta_r$, $\overline\zeta_r$, $\overline\zeta_r$. Equation (\ref{equation: trace}) leaves three possibilities for the eigenvalues of $\theta_M(A)$: they must be either $1, 1, \zeta_r$, $\overline \zeta_r$, or $1$, $1$, $\zeta_r$, $\zeta_r$, or $1$, $1$, $\overline\zeta_r$, $\overline\zeta_r$. Since $\Tr\theta_M(A)$ is rational (by proposition \ref{proposition:norm}), only the first possibility can occur.
\end{proof}

\subsection{Equidistribution for Frobenius conjugacy classes}\label{section: Frob conj classes equidist}

We first recall the well-known notion of equidistribution on a compact topological space $X$ (see \cite[Chap.\ 1]{Se68}).
Let $\mathcal C(X)$ denote the Banach space of continuous, complex valued functions $f$ on $X$, with norm
$||f||=\sup_{x\in X}|f(x)|$.
Let $\mu$ be a Radon measure on $X$, a continuous linear form on $\mathcal C(X)$. 
Let $\{x_i\}_{i\geq 1}$ be a sequence of points of $X$. The sequence
$\{x_i\}_{i\geq 1}$ is said to be equidistributed with respect to $\mu$ if for
every $f\in \mathcal C(X)$ we have
$$\mu(f)=\lim_{m\rightarrow \infty}\frac{1}{m}\sum_{i=1}^mf(x_i)\,.$$

Note that if $\{x_i\}_{i\geq 1}$ is equidistributed with respect to $\mu$, then
$\mu$ is positive and has total mass $1$. We are particularly interested in the case where $X$ is an interval $I$ of $\R$.
In this case the $n$th moment $\M_n[\mu]$ of $\mu$ is the value $\mu(\varphi_n)$, where $\varphi_n$ is the function of $\mathcal C(I)$ defined by $\varphi_n(z)=z^n$.
Analogously, the $n$th moment of a sequence $\{x_i\}_{i\geq 1}$ on $I$, if it exists, is defined by
$$
\M_n[\{x_i\}_{i\geq 1}]=\lim_{m\rightarrow \infty}\frac{1}{m}\sum_{i=1}^m x_i^ n\,.
$$
Thus  if the sequence $\{x_i\}_{i\geq 1}$ is equidistributed with respect to $\mu$ on $I$, then its $n$th moment exists and is equal to $\M_n[\mu]$.

Let $F/k$ be a field extension, and let $P_{E_F}$ denote the set of primes of $F$ at which $E_F$ has good reduction. We write the normalized $L$-polynomial for $E_F$ at a prime $\p$ of $P_{E_F}$ as
$$
\overline L_{\p}(E_F,T)=1+a_1(E_F)(\p)T+T^2\,.$$

Choose an \emph{ordering by norm} $\{\p_i\}_{i\geq 1}$ of $P_{E_F}$, that is, an ordering for which $\norm\p_i\leq \norm\p_j$ for all $1\le i\le j$, and let $a_1(E_F)$ denote the sequence  
$$
\{a_1(E_F)(\p_i)\}_{i\geq 1}
$$
of real numbers in the interval $[-2,2]$. Equidistribution statements about $a_1(E_F)$ do not depend on the particular ordering by norm we have chosen. 

Until the end of this section, we assume that $F$ contains $kM$.
We begin by recalling classical results of Hecke and Deuring that yield equidistribution for $a_1(E_F)$ with respect to the measure
$$\mu_{\rm{cm}}=\frac{1}{\pi}\frac{dz}{\sqrt{4-z^2}}\,,$$ supported on $[2,-2]$.
Here $dz$ denotes the restriction of the Lebesgue measure on $\R$ to the interval $[-2,2]$. The measure $\mu_{\rm{cm}}$ is uniquely characterized by the fact that it is continuous and its $n$th moment is $b_n:=b_{0,n}$ (as in Theorem~\ref{theorem: main result}).

We actually require a slightly stronger equidistribution statement than the one above. Let $c$ be a Frobenius conjugacy class of an arbitrary finite Galois extension $F'/F$, and let $P_{c}$ denote the set of primes in $P_{E_F}$ that are unramified in $F'$ and whose Frobenius conjugacy class is $c$. We will show that the subsequence $a_{1,c}(E_F)$ of $a_1(E_F)$ obtained by restricting to the primes in $P_{c}$ is also equidistributed with respect to~$\mu_{\rm{cm}}$.

\begin{remark}\label{remark: cm case}
Henceforth, for a compact group $G$, let $\mu(G)$ denote its Haar measure. In terms of the (generalized) Sato-Tate Conjecture, the measure $\mu_{\rm {cm}}$ is seen as $\Phi_{1,*}(\mu(\ST(E_F)))$, where $\Phi_1$ is the trace map defined in equation (\ref{equation: projection map}) and $\ST(E_F)=\Unitary(1)$. Recall that the Sato-Tate group $\ST(E)$ of an elliptic curve~$E$ defined over $k$ with CM by $M$ is $\Unitary (1)$ (embedded in $\SU(2)$) if $M$ is contained in~$k$, and the normalizer of $\Unitary (1)$ in $\SU(2)$ if $M$ is not contained in~$k$.
\end{remark}

We follow the presentation in \cite[Chap.\ 1]{Gr80}. Let $\mathfrak p$ be a prime of $F$ of good reduction  for $E_F$. Let $\overline\F_{\p}$ denote the algebraic closure of the residue field of $F$ at $\p$. The image of the injection
$$
\End(A_{\overline\Q})\otimes\Q=M\hookrightarrow \End(E_{\overline\F_{\p}})\otimes \Q
$$
contains the Frobenius endomorphism $\Fr_{\p}\colon E_{\overline\F_{\p}}\rightarrow E_{\overline\F_{\p}}$, which acts on a point by raising its coordinates to the $q$th  power, where $q=\norm\p$.
Let $\alpha(\p):=\alpha(E_F)(\p)\in M^*$ denote the preimage of $\Fr_{\p}$ under this injection.
Since the characteristic polynomial of $\Fr_\p$ is reciprocal to the $L$-polynomial of $E_F$ at $\p$, we have
\begin{equation}\label{equation: prou}
a_1(E_F)(\mathfrak p)=-\frac{1}{\norm \p^{1/2}}\bigl(\sigma(\alpha(\p))+\overline\sigma(\alpha(\p))\bigr)\,.
\end{equation}

For any place $v$ of $ F$, let $ F_v$ denote the completion of $ F$ at $v$ and let $\mathcal O_v$ denote the ring of integers of $ F_v$. Let $I_{ F}=\prod_v' F_v$ denote the group of ideles of~$F$.
Here the product runs over all places $v$ of $ F$, and the prime means that if $\s=(\s_v)$ belongs to $I_F$, then $\s_v$ is in $\mathcal O_v^*$ for all but finitely many $v$.
We write~$v_\p$ for the valuation associated to a finite prime $\p$ of $ F$. We then attach to $E_F$ the group homomorphism
$$\chi_{E_F}\colon I_F\rightarrow M^*$$
uniquely characterized by the following three properties:
\begin{enumerate}[(i)]
\item $\Ker(\chi_{E_F})$ is an open subgroup of $I_F$.
\item If $\s=(a)$ is a principal idele ($a\in F^*$), then $\chi_{E_F}(\s)=N_{F/M}(a)$.
\item If $\s=(\s_v)$ is an idele with $\s_v=1$ at all infinite places of $F$ and at those finite places where $E_F$ has bad reduction, then
$$
\chi_{E_F}(\s)=\prod_{v_\p}\alpha(\p)^{v_\p(\s_\p)}\,.
$$
\end{enumerate}

\subsubsection{The $1$-dimensional $\ell$-adic representation attached to $E_F$.}
Fix a prime~$\ell$ different from the characteristic of $\overline\F_\p$ and an embedding of $\Qbar$ into $\Qbar_\ell$, and let $V_\ell(E_F)$ denote the (rational) $\ell$-adic Tate module of $E_F$. Define
\begin{equation}\label{equation: tensor product}
V_\sigma(E):=V_\ell(E_F)\otimes_{M,\sigma}\Qbar_\ell\,,
\end{equation}
where the tensor product is taken via the embedding $M\hookrightarrow\Qbar_\ell$ induced by $\sigma$. Similarly define $V_{\overline\sigma}(E)$.  We then have an isomorphism of $\overline\Q_\ell[G_{F}]$-modules:
\begin{equation}\label{equation: decomp ladic} 
V_\ell(E_{F})\otimes \overline\Q_\ell\simeq V_\sigma(E)\oplus V_{\overline\sigma}(E)\,.
\end{equation} 
Let $\varrho_{\ell,\sigma}\colon G_F\rightarrow \Aut(V_\sigma(E))$ denote
the $\ell$-adic character corresponding to the action of $G_F$ on $V_\sigma(E)$. If $\Frob_\p$ is an arithmetic Frobenius at $\p$ in $G_F$, then the value of $\varrho_{\ell,\sigma}(\Frob_\p)$ is $\sigma(\alpha(\p))$.
Define
$$
\psi_{\ell,\sigma}\colon I_F\rightarrow (M\otimes_{M,\sigma}\overline \Q_\ell)^*,\qquad \psi_{\ell,\sigma}(\s)=\chi_{E_F}(\s)\otimes (N_{F/M}(\s^{-1}))_\ell\,,
$$
where for an idele $\s$ in $ I_F$, the component of the idele $N_{F/M}(\s)$ in $I_M$ corresponding to the place $w$ is  $\prod_{v|w}N_{F_v/M_w}(\s_v)$, where the product runs over all places~$v$ of $F$ lying over $w$.
We then have $\psi_{\ell,\sigma}(F^*)=1$, by property (ii).
Thus  $\psi_{\ell,\sigma}$ is a continuous character on the group $C_F=I_F/F^*$ of classes of ideles. Since its image is totally disconnected, it is a character of $C_F/C_F ^0$, where $C_F^0$ is the identity component of $C_F$. Artin reciprocity yields an isomorphism
$\Rec\colon G_F^{\ab}\rightarrow C_F/C_F^0\,.$
Property (iii) then implies that $\psi_{\ell,\sigma}\circ\Rec(\Frob_\p)=\sigma(\alpha(\p))$, thus
\begin{equation}\label{equation: equal characters}
\psi_{\ell,\sigma}\circ\Rec(\Frob_\p)=\varrho_{\ell,\sigma}\,,
\end{equation}
as $\ell$-adic characters of $G_F$.

\subsubsection{The Hecke character attached to $E_F$.}
A \emph{Hecke character} of $F$ is a continuous homomorphism $\psi\colon I_{ F}\rightarrow \C^*$ such that $\psi(F^*)=1$.
For primes $\p$ where $\psi$ is unramified, let $\psi(\p)$ denote $\psi(\s)$, where $\s_\p$ is a uniformizer of $\mathcal O_\p$ and $\s_v=1$ for $v\ne \p$, and let $\psi(\p)=0$ when $\psi$ is ramified at $\p$.
 The $L$-function of $\psi$ is defined as
$$
L(\psi,s):=\prod_\p(1-\psi(\p)\norm \p^{-s})^{-1}\,.
$$
Hecke \cite{He20} showed that if $\psi$ is nontrivial, then $L(\psi,s)$ is a nonzero holomorphic function for $\Re(s)\geq 1$.
Let us fix an embedding of $\Qbar$ into $\C$, so that we may view $\sigma$ and $\overline\sigma$ as embeddings of $M$ into $\C$.
Define
$$
\psi_{\infty,\sigma}\colon I_F\rightarrow (M\otimes_{M,\sigma}\C)^*,\qquad \psi_{\infty,\sigma}(\s)=\chi_{E_F}(\s)\otimes (N_{F/M}(\s^{-1}))_\infty\,,
$$
where $\infty$ denotes the only infinite place of $M$.
Property (ii) of $\chi_{E_F}$, implies that $\psi_{\infty,\sigma}$ is a Hecke character. It is unramified at the primes of good reduction for~$E_F$, and we note that $\overline\psi_{\infty,\sigma}=\psi_{\infty,\overline\sigma}$. Let $|z|$ denote the absolute value
of a complex number $z$ and define
$$
\psi_{\infty,\sigma}^1\colon I_F\rightarrow \Unitary(1),\qquad \psi_{\infty,\sigma}^1(\s)=\psi_{\infty,\sigma}(\s)/|\psi_{\infty,\sigma}(\s)|\,.
$$
For every prime $\p$ of good reduction for $E_F$, let
\begin{equation}\label{eq: alpha1}
\alpha_1(\p):=\alpha_1(E_F)(\p):=\psi_{\infty,\sigma}^1\circ\Rec(\Frob_\p)=\sigma(\alpha(\p))/\norm \p^{1/2}.
\end{equation}
Let $\alpha_1$ denote the sequence $\{\alpha_1(\p_i)\}_{i\geq 1}$.

\subsubsection{Equidistribution statements.}
For a finite Galois extension $F'/F$ and a conjugacy class $c$ of $\Gal(F' /F)$, let $P_c$ be as above.
Let $\alpha_{1,c}:=\alpha_{1,c}(E_F)$ denote the subsequence of $\alpha_1$ obtained by restricting to the primes of $P_c$.
Our goal is to prove the following proposition. 

\begin{proposition}\label{proposition: equidistribution classes}
Let $c$ be any conjugacy class of $\Gal(F'/F)$. Then $\alpha_{1,c}$ is equidistributed with respect to $\mu(\Unitary(1))$.
\end{proposition} 

We first recall a theorem of Serre. Let $G$ be a compact group and $X$ be the set of its conjugacy classes. Let $P$ be an infinite subset of the primes of~$F$, and let $\{\p_i\}_{i\geq 1}$ be an ordering by norm of $P$.
Assume that each $\p$ of $P$ has been assigned a corresponding element $x_\p$ in $X$.

\begin{theorem}{\cite[p.\ I-23]{Se68}}\label{theorem: Serre equid} The sequence $\{x_{\p_i}\}_{i\geq 1}$ is equidistributed over~$X$ with respect to the image on $X$ of the Haar measure of $G$ if and only if $L(\varrho,s)$ is holomorphic and nonzero at $s=1$ for every irreducible and nontrivial representation $\varrho$ of $G$. Here $L(\varrho,s)$ stands for the infinite product
$$
\prod_{\p\in P} \det(1-\varrho(x_\p)\norm \p^{-s})^{-1}\,.
$$
\end{theorem}

We now use Theorem~\ref{theorem: Serre equid} to prove Proposition~\ref{proposition: equidistribution classes}.
\begin{proof}
We first reduce to the case that $F'/F$ is abelian (in fact, cyclic).
Let $\tau$ be an element of $c$, and let $f$ denote its order. Define
$$
I(\tau)=\{i\in\{0,1,\dots,f-1\}\ |\  [\tau^i]=c\}\,.
$$
Let $H$ be the subfield of $F'$ fixed by ${\langle \tau\rangle}$. The residue degree over $F$ of a prime~$\pp$ of $H$ lying over $\p$ is $1$, thus $\alpha_1(E_H)(\pp)=\alpha_1(E_F)(\p)$. Then $\alpha_{1,c}(E_F)$ is the union 
$$
\bigsqcup_{i\in I(\tau)} \alpha_{1,\tau^i}(E_H)\,,
$$
where we identify $\tau^i$ with its conjugacy class in the cyclic group $\Gal(F'/H)$.
It follows that the sequence $\alpha_{1,c}=\alpha_{1,c}(E_F)$ is $\mu(\Unitary(1))$-equidistributed if all its subsequences $\alpha_{1,\tau^i}(E_H)$ are. But this hypothesis is true if one assumes that the proposition holds for abelian extensions.

So suppose that $F'/F$ is abelian, and define $G:=\Unitary(1)\times\Gal(F'/F)$ and $x_\p:=\alpha_1(\p)\times \Frob_\p$ for each prime in $P_{E_F}$ unramified in $F'/F$. 
Proving the proposition is equivalent to showing that $\{x_{\p_i}\}_{i\geq 1}$ is equidistributed over the set $X$ of conjugacy classes of $G$ with respect to the measure induced by the Haar measure of $G$.
The irreducible characters of $G$ are of the form $\phi_a\otimes \chi$, where $\phi_a\colon \Unitary(1)\rightarrow \C^*$ is a character of $\Unitary(1)$, which is of the form  $\phi_a(z)=z^a$ for some integer $a$, and $\chi$ is an irreducible character of $\Gal(F'/F)$, which is $1$-dimensional, since $\Gal(F'/F)$ is abelian. By Theorem~\ref{theorem: Serre equid}, it is enough to show that if $\phi_a\otimes\chi$ is nontrivial, then
$$
L(\phi_a\otimes\chi,s)=\prod_\p(1-\frac{\psi_{\infty,\sigma}(\p)^a}{\norm \p^{a/2}}\chi(\p)\norm \p ^{-s})^{-1}
$$
is holomorphic and nonzero at $s=1$.

Suppose that $a\geq 0$. Via Artin reciprocity we may view $(\psi_{\infty,\sigma})^a\otimes \chi$ as a Hecke character, and then $L(\phi_a\otimes\chi,s)$  is equal, up to a finite number of factors, to the Hecke $L$-function $L((\psi_{\infty,\sigma})^a\otimes\chi,s+\frac{a}{2})$.
Since $(\psi_{\infty,\sigma})^a\otimes \chi$ is a nontrivial Hecke character, its $L$-function is holomorphic and nonzero for $\Re s\geq 1$ and in particular, so is $L(\phi_a\otimes\chi,s)$ at $s=1$. If $a<0$, then we can repeat the argument, observing that  $L(\phi_a\otimes\chi,s)$ also coincides, up to a finite number of factors, with the Hecke $L$-function $L((\overline\psi_{\infty,\sigma})^{-a}\otimes\chi,s-\frac{a}{2})$.
\end{proof}

Recalling that $\mu_{\rm{cm}}=\frac{1}{\pi}\frac{dz}{\sqrt{4-z^2}}$ supported on $[-2,2]$ is the image by $\Phi_1$ of the Haar measure of $\Unitary(1)$, we obtain the following.

\begin{corollary}\label{corollary: equidist frobenius classes} Let $E$ be an elliptic curve defined over $k$ with CM by an imaginary quadratic field $M$. Let $F$ be any field containing $kM$, let $F'/F$ be a finite Galois extension, and let $c$ be a conjugacy class of $\Gal(F'/F)$.
Then
\begin{enumerate}[\rm (i)]
\item The sequence $a_{1,c}(E_F)$ is equidistributed with respect to the measure $\mu_{\rm{cm}}$.
\item $\M_n[a_{1,c}(E_F)] = \M_n[a_1(E_F)]$.
\end{enumerate}
\end{corollary}

\subsection{Equidistribution of \texorpdfstring{$a_1(A)$}{} and \texorpdfstring{$a_2(A)$}{}}\label{section: moments general}

As in \S\ref{section: Galois modules}, $A$ is an abelian surface defined over $k$ with $A_\Qbar\sim E^2_\Qbar$, where $E$ is an elliptic curve defined over $k$ with CM by $M$, and we have the tower of fields $kM\subseteq K\subseteq L$, where $L/k$ is the minimal extension over which all the homomorphisms from $A_\Qbar$ to $E_\Qbar$ are defined, and $K/k$ is the minimal extension over which all the endomorphisms of $A_\Qbar$ are defined.

For any field extension $F/k$, let $P_{A_F}$ denote the set of primes of $F$ at which~$A_F$ has good reduction. For~$\p$ in $P_{A_F}$, we write the normalized $L$-polynomial for $A_F$ at $\p$ as
$$
\overline L_{\mathfrak p}(A_F,T)=1+ a_1(A_F)(\mathfrak p)T+ a_2(A_F)(\mathfrak p)T^2+ a_1(A_F)(\mathfrak p)T^3+T^4\,.
$$

Let $P$ be the set of primes lying in $P_{A_F}$ and $P_{E_F}$ that are unramified in $FL$. Choose an ordering $\{\p_i\}_{i\geq 1}$ by norm of $P$, and let $a_1(A_F)$ and $a_2(A_F)$ denote the sequences
$$
\{a_1(A_F)(\p_i)\}_{i\geq 1}\,,\qquad\{a_2(A_F)(\p_i)\}_{i\geq 1}\,
$$
respectively.
In this section we use the results in  \S\ref{section: Galois modules} and \S\ref{section: Frob conj classes equidist} to prove equidistribution for $a_1(A)$ and $a_2(A)$.

\begin{lemma}\label{lemma: trans coeff}
Let $\p$ be a prime of good reduction for $A$ and $E$ that splits in $kM$ and is unramified in $L$.

\begin{enumerate}[\rm (i)]
\item With $u_1=\operatorname {Re} a_1(\theta)(\Frob_{\mathfrak p})$ and $u_2=\operatorname{Im} a_1(\theta)(\Frob_{\mathfrak p})$ we have
$$a_1(A)(\mathfrak p)=u_1a_1(E)(\p) \pm u_2 \sqrt{4-a_1(E)(\p)^2}\,.$$
\item With $v_1=\operatorname {Re} a_2(\theta)(\Frob_{\mathfrak p})$ and $v_2=\operatorname{Im} a_2(\theta)(\Frob_{\mathfrak p})$ we have
 $$a_2(A)(\mathfrak p)=v_1 a_1(E)(\p)^2-2v_1+|a_1(\theta)(\Frob_{\mathfrak p})|^2\mp v_2a_1(E)(\p)\sqrt{4-a_1(E)(\p)^2}\,. $$
\end{enumerate}
\end{lemma}

\begin{proof}
Define $V_{\sigma}(A)$ and $V_{\overline\sigma}(A)$ as in equation (\ref{equation: tensor product}).
We then have the following isomorphism of $\overline\Q_\ell[G_{kM}]$-modules:
$$V_\ell(A_{kM})\otimes\overline\Q_\ell\simeq V_\sigma(A)\oplus V_{\overline\sigma}(A)\,.$$
By arguments analogous to those in \cite[Thm.~3.1]{Fit10}, we have
$$V_\sigma(A)\simeq \theta_{M,\sigma}(E,A)\otimes V_\sigma(E)\,,\qquad V_{\overline\sigma}(A)\simeq \theta_{M,\overline\sigma}(E,A)\otimes V_{\overline\sigma}(E)\,.$$
Thus there is an isomorphism of $\overline\Q_\ell[G_{kM}]$-modules:
\begin{equation}\label{equation:flip}
V_\ell(A)\otimes\overline\Q_\ell\simeq \theta_{M,\sigma}(E,A)\otimes V_\sigma(E)\,\oplus\,\theta_{M,\overline\sigma}(E,A)\otimes V_{\overline\sigma}(E)\,.
\end{equation}

To shorten notation, we write $\alpha_1(\p)$ for $\alpha_1(E_{kM})(\p)=\sigma(\alpha(E_{kM})(\p))/\norm{\p}^{1/2}$, as defined in \eqref{eq: alpha1}. Then $\overline{ \alpha_1(\p)}=\overline\sigma(\alpha(E_{kM})(\p))/\norm{\p}^{1/2}$, and (\ref{equation:flip}) implies that
\begin{equation}\label{equation: determines coefficients}
\begin{array}{l}
\vspace{6pt}
a_1(A_{kM})(\mathfrak p)=-a_1(\mathfrak p)\alpha_1(\p)-\overline{a_1(\mathfrak p)}\,\overline{\alpha_1(\p)}\,,\\
a_2(A_{kM})(\mathfrak p)= a_2(\mathfrak p) \alpha_1(\p)^2+\overline{a_2(\mathfrak p)}\,\overline{\alpha_1(\p)}^2+a_1(\mathfrak p)\overline{a_1(\mathfrak p)}\,,\\
\end{array}
\end{equation}
where $a_i(\p)$ denotes $a_i(\theta)(\Frob_\p)$.
The proposition then follows from the fact that $a_1(E_{kM})(\p)= -\alpha_1(\p) -\overline {\alpha_1(\p)}$.
\end{proof}

\begin{proposition}\label{proposition:import} For $\tau\in\Gal(L/kM)$, let $u=u(\tau)=|a_1(\theta)(\tau)|$. Then $a_1(A_{kM})$ and  $a_2(A_{kM})$ are equidistributed with respect to the measures
\begin{enumerate}[\rm (i)]
\item $\mu(a_1(A_{kM})):=\frac{1}{[L:kM]} \frac{1}{\pi}\sum_\tau\frac{dz}{\sqrt{4u^2-z^2}}\mathbf{1}_{[-2u,2u]}$\,,
\item $\mu(a_2(A_{kM})):=\frac{1}{[L:kM]}\frac{1}{\pi}\sum_\tau \frac{dz}{\sqrt{4-(u^2-z)^2}}\mathbf{1}_{[u^2-2,u^2+2]} \,,$
\end{enumerate}
whose support lies in the intervals $I_1=[-4,4]$ and $I_2=[-6,6]$, respectively.
In each sum $\tau$ ranges over $\Gal(L/kM)$, and $\mathbf 1_{[a,b]}$ is the characteristic function of the interval $[a,b]\subseteq \R$.  Moreover, we have
\begin{enumerate}[\rm (i)]
\item $\M_n[ a_1(A_{kM})]=\frac{1}{[L:kM]}\sum_\tau b_{0,n}u^n\,,$
\item $\M_n[ a_2(A_{kM})]=\frac{1}{[L:kM]}\sum_\tau b_{u^2,n}\,,$
\end{enumerate}
where the integer $b_{m,n}$ is the coefficient of $X^n$ in $(X^2+mX+1)^n$.
\end{proposition}
\begin{proof}
We can rewrite the equations in (\ref{equation: determines coefficients}) as follows:
$$
\begin{array}{lll}
a_1(A_{kM})(\p)&=&|a_1(\p)|\left(\frac{-a_1(\p)}{|a_1(\p)|}\alpha_1(\p)+\frac{-\overline{a_1(\p)}}{|a_1(\p)|}\ \overline{\alpha_1(\p)}\right)\,,\\[4pt]
a_2(A_{kM})(\p)&= &|a_2(\p)|\left(\frac{a_2(\p)^{1/2}}{|a_2(\p)|^{1/2}} \alpha_1(\p)+\frac{\overline{a_2(\theta)}^{1/2}}{|a_2(\p)|^{1/2}}\ \overline{\alpha_1(\p)}\right)^2-2|a_2(\p)|+|a_1(\p)|^2\\
& =&\left(a_2(\p)^{1/2} \alpha_1(\p)+\overline{a_2(\p)}^{1/2}\ \overline{\alpha_1(\p)}\right)^2-2+|a_1(\p)|^2\,,
\end{array}
$$
where $a_i(\p)$ denotes $a_i(\theta)(\Frob_\p)$, and we have used $|a_2(\p)|=1$.
The equidistribution statements now follow from the Cebotarev density theorem and the two facts below:

\begin{enumerate}[\rm (1)]
\item For any $z\in\Unitary(1)$ and any conjugacy class $c$ of $\Gal(L/kM)$, the sequence $z\alpha_{1,c}$ is $\mu(\Unitary(1))$-equidistributed on $\Unitary(1)$. Indeed, Proposition \ref{proposition: equidistribution classes} ensures equidistribution of $\alpha_{1,c}$, and invariance under translations is in fact the defining property of the Haar measure.
Thus the sequence $z\alpha_{1,c}+\overline z\overline\alpha_{1,c}$ is $\mu_{\rm cm}$-equidistributed on $I_1(E_{kM})=[-2,2]$.
\item If a sequence $\beta=\{\beta_i\}_{i\geq 1}$ is $\mu_{\rm cm}$-equidistributed on $[-2,2]$, then for $u\in\mathbb R_{>0}$
\begin{itemize}
\item The sequence $u\beta$ is equidistributed on $[-2u,2u]$ with respect to the measure $\frac{1}{\pi}\frac{dz}{\sqrt{4u^2-z^2}}$. 
\item The sequence $\{\beta_i^2-2+u^2\}_{i\ge 1}$ is equidistributed on $[u^2-2,u^2+2]$ with respect to the measure $\frac{1}{\pi}\frac{dz}{\sqrt{4-(u^2-z)^2}}$.
\end{itemize}
\end{enumerate}
Regarding the moments, the Cebotarev density theorem implies that
$$\M_n[a_1(A_{kM})]=\frac{1}{[L:kM]}\sum_\tau |a_1(\theta)(\tau)|^{n}\cdot \M_n[z([\tau])\alpha_1+\overline{z([\tau])}\overline{\alpha}_1\ |\ P_{[\tau]}]\,,$$
where $z([\tau])=-a_1(\theta)(\tau)/|a_1(\theta)(\tau)|$. But now (1) implies that 
$$
\M_n[z([\tau])\alpha_1+\overline{z([\tau])}\overline{\alpha}_1\ |\ P_{[ \tau]}]=b_{0,n}\,.
$$

The same argument is used to compute 
$$\M_n[a_2(A_{kM})]=\frac{1}{[L:kM]}\sum_\tau\sum_{i=0}^{n}\binom{n}{i}\binom{2i}{i}\left(|a_1(\theta)(\tau)|^2-2\right)^{n-i}\,.$$
One then applies
\begin{align*}
\sum_{i=0}^n\binom{n}{i}\binom{2i}{i}(m-2)^{n-i} &= [X^n]((X+1)^2+(m-2)X)^n\\
 &= [X^n](X^2+mX+1)^n=b_{m,n},
\end{align*}
where $[X^n](f(X)$ denotes the coefficient of $X^n$ in the polynomial $f(X)$.
\end{proof}

We now generalize the definitions of $o(r)$ and $\overline o(s)$ given in \S\ref{subsection: main result} for $k=\Q$.

\begin{definition} Let $o(r)$ count the elements in $\Gal(L/kM)$ whose projection in $\Gal(K/kM)$ has order $r$.
Let $\overline o(s)$ count the elements in $\Gal(L/k)\setminus \Gal(L/kM)$ of order $s$.
\end{definition}

If $k=kM$, then $a_i(A)$ is equidistributed with respect to $\mu(a_i(A_{kM}))$ and $\M_n[a_i(A)]=\M_n[a_i(A_{kM})]$, for $i=1,2$.

\begin{corollary}\label{corollary: a2 moments}
Suppose $k\not=kM$. Then the $a_1(A)$ and $a_2(A)$ are equidistributed with respect to the measures
\begin{enumerate}[\rm (i)]
\item $\mu(a_1(A)):=\frac{1}{2}\mu(a_1(A_{kM}))+\frac{1}{2}\delta_0\,,$
\item $\mu(a_2(A)):=\frac{1}{2}\mu(a_1(A_{kM}))+\frac{1}{2[L:kM]}\bigl(\bar o(2)\delta_2 + \bar o(4)\delta_{-2} + \bar o(6)\delta_{-1} + \bar o(12)\delta_{1}\bigr)\,,$
\end{enumerate}
whose support lies in the intervals  $I_1=[-4,4]$ and $I_2=[-6,6]$, respectively. Here $\delta_z$ denotes the Dirac measure at $z$.
We also have
$$
\begin{array}{lll}
{\rm (i)} & \M_n[a_1(A)]=&\frac{1}{[L:k]}\bigl(o(1)2^n+o(3) + o(4)2^{n/2} + o(6)3^{n/2}\bigr) b_{0,n}\,,\\[6pt]
{\rm (ii)} &\M_n[a_2(A)]=&\frac{1}{[L:k]}\bigl(o(1)b_{4,n} + o(2)b_{0,n} + o(3)b_{1,n} + o(4)b_{2,n} + o(6)b_{3,n} \\[6pt]
&& \qquad\quad +\  \bar o(2)2^n + \bar o(4)(-2)^n + \bar o(6)(-1)^n + \bar o(12))\,.
\end{array}
$$
\end{corollary}

\begin{proof}
We focus on the proof of the statements about the moments, since the arguments involved suffice to deduce the statements about the measures.
Statement (i) follows from Propositions \ref{proposition:norm}, \ref{proposition:trtheta}, and ~\ref{proposition:import},  and the equality $$\M_{2n}[ a_1(A_{kM})]=2\cdot \M_{2n}[ a_1(A)]\,,$$
which follows from the fact that if $\p$ is a prime of $k$ where $A$ has good reduction and $\p$ is inert in $kM$, then $A$ is supersingular at $\p$ and $a_1(A)(\p)=0$.

For (ii), let $\nu$ denote the nontrivial conjugacy class of $\Gal(kM/k)$.
Note that
$$\M_n[a_2(A)]=\frac{1}{2}\M_n[a_2(A)\ |\ P_{\id}]\ +\ \frac{1}{2}\M_n[a_2(A)\ |\ P_{\nu}].$$
To compute $\M_n[a_2(A)\ |\ P_{\id}]=\M_n[a_2(A_{kM})]$, we apply Proposition \ref{proposition:import}.
We then claim that
$$\M_n[a_2(A_k)\ |\ P_{\nu}]=\frac{1}{[L\colon kM]}\bigl(\overline{o}(2)2^n + \overline{o}(4)(-2)^n+\overline{o}(6)(-1)^n+\overline{o}(12)\bigr)\,.$$
We may restrict to primes $\p$ of $k$ that are inert in $kM$, of absolute residue degree~1, and of good reduction for both $A$ and $E$.
The polynomial $\overline L_{\p}(A,T)$ must then be one of the five listed in (\ref{equation: five possibilities}).

We now consider the Rankin-Selberg polynomial $\overline L_{\p}(E,\theta_\Q(E,A),T)$, whose roots are all products of roots of $\overline L_{\p}(E,T)=1+T^2$, and all roots of the  polynomial  $\det(1-\theta_\Q(E,A)(\Frob_{\p})T)$.
More explicitly, if $s$ is the order of $\Frob_{\p}$ in $\Gal(L/k)$, one may apply Proposition \ref{proposition:trthetaQ} to compute $\overline L_{\p}(E,\theta_\Q(E,A),T)$.  This yields:
\vspace{-8pt}
\begin{center}
\begin{tabular}{llllllll}
$s=2$: & $(1+T^2)^4$ && $s=6$: & $(1-T^2+T^4)^2$ && $s=12$: & $(1+T^2+T^4)^2$\vspace{2pt}\\
$s=4$: & $(1-T^2)^4$ && $s=8$: & $(1+T^4)^2$\vspace{2pt}\\
\end{tabular}
\end{center}

By arguments analogous to those of \cite[theorem 3.1]{Fit10}, there is an inclusion of $\Q_\ell[G_k]$-modules
$$
V_\ell(A)\subseteq V_\ell(E)\otimes \theta_\Q(E,A)\,.
$$
This implies that $\overline L_{\p}(A,T)$ divides $\overline L_{\p}(E,\theta_\Q(E,A),T)$.
It immediately follows that $\overline L_{\mathfrak p}(A,T)$ is
\begin{center}
\begin{tabular}{llllllll}
$s=2$: & $(1+T^2)^2$ & $\,\,$ & $s=6$: & $1-T^2+T^4$ & $\,\,$ & $s=12$: & $1+T^2+T^4$\vspace{2pt}\\
$s=4$: & $(1-T^2)^2$ & $\,\,$ & $s=8$: & $1+T^4$\vspace{2pt}\\
\end{tabular}
\end{center}
Finally, we observe that the condition $\overline L_{\mathfrak p}(A,T)$ divides $\overline L_{\mathfrak p}(E,\theta_\Q(E,A),T)$ implies that $s$ can not attain any value other than the ones considered.\end{proof}

\subsection{Additional remarks}\label{subsection: additional}

As noted in the introduction, all 32 of the genus 2 Sato-Tate groups with identity component isomorphic to $\Unitary(1)$ can arise as the Sato-Tate group of an abelian variety~$A$ defined over~$k$ with $A_\Qbar\sim E^2_\Qbar$, where $E$ is an elliptic curve defined over $k$ (with CM).

However, not all 10 of the genus 2 Sato-Tate groups with identity component isomorphic to $\SU(2)$, can arise as the Sato-Tate group of an abelian variety $A$ defined over~$k$ such that $A_\Qbar\sim E^2_\Qbar$\,, where $E$ is an elliptic curve defined over~$k$ (without CM).\footnote{All Sato-Tate groups with identity component $\SU(2)$ can occur for an $A$ over $k$ such that $A_\Qbar\sim E^2_\Qbar$ for \emph{some} elliptic curve $E$, but this curve need not be defined over $k$.}
The Sato-Tate groups for which this is not true are the four whose component group contains an element of order 4 or 6.
Indeed, recall that $\theta_\Q(E,A)$ (resp.\ $\theta_\Q(A)$) is the representation afforded by $\Hom(E_L,A_L)\otimes\Q$ (resp.\ $\End(A_L)\otimes\Q$).
As in the proof of Proposition \ref{proposition:norm}, one can then show that $\theta_\Q(A)=\theta_\Q(E,A)^{\otimes 2}$, that is, $a_1(\theta_\Q(A))=a_1(\theta_\Q(E,A))^2$. But if $\tau\in\Gal(K/k)$ has order $4$ or $6$ then $a_1(\theta_\Q(A))(\tau)=2$ or $3$, which are not squares in $\Q$.

We end this section by computing the density $z_1(A_k)$ of zero traces of an abelian variety $A$ defined over $k$ such that $A_\Qbar\sim E^2_\Qbar$ for some elliptic curve~$E$ defined over $\Qbar$.

\begin{lemma} Let $A$ be an abelian variety defined over $k$ such that $A_\Qbar\sim E^2_\Qbar$\,, where $E$ is an elliptic curve defined over $\Qbar$ (not necessarily over $k$). Let $M$ denote the CM field if $E$ has CM, and let $M=\Q$ otherwise.
Then
$$
z_1(A_k)=
\begin{cases}
\frac{o(2)}{|\Gal(L/kM)|} & \text{if $[kM\colon k]=1$\,,}\\[6pt]
\frac{1}{2}+\frac{1}{2}\frac{o(2)}{|\Gal(L/kM)|} & \text{if $[kM\colon k]=2$\,.}
\end{cases}
$$
\end{lemma}

\begin{proof}
Except for a set of density zero, any prime $\p$ of $k$ that does not split in $kM$ is supersingular, in which case $a_1(A)({\p})=0$.
This gives density $0$ in the first case and density $\frac{1}{2}$ in the second case.
Among the primes that split in $kM$, we wish to show that exactly the proportion $o(2)/|\Gal(L/kM)|$ have trace $0$.
Among these primes, the density of the supersingular primes is zero.
Let $\p$ be a non-supersingular prime of good reduction for $A$ that splits in $kM$.
From Remark 4.8 in \cite{FKRS} in the non-CM case, and from Proposition \ref{proposition:trtheta} in the CM case, the roots of $L_{\p}(A,T)$ are $\alpha$, $\overline\alpha$, $\zeta_r\alpha$, $\overline\zeta_r\overline\alpha$, where $r$ is the order of $\Frob_{\p}$ in $\Gal(K/k)$ and where $\alpha/\overline\alpha$ is not a root of unity.
It follows that $\alpha+\overline\alpha+\zeta_r\alpha+\overline\zeta_r\overline\alpha=0$ if and only if $r=2$.
One then applies the Cebotarev density theorem.
\end{proof}

\section{Twists of \texorpdfstring{$y^2=x^5-x$}{} and \texorpdfstring{$y^2=x^6+1$}{}}\label{section: curves}

In this section, we strengthen  the results of \S\ref{section: squares} in the particular case that $k=\Q$ and $A\sim_\Q \Jac(C)$, where $C$ is a twist of the curve $y^2=x^5-x$ or $y^2=x^6+1$.
We first introduce some convenient notation. Let $C^0_2$ and $C^0_3$ denote the curves defined over $\Q$ by the equations
$$
C^0_2\colon y^2=x^6-5x^4-5x^2+1,\qquad C^0_3\colon y^2=x^6+1\,.
$$
The curve $C^0_2$ is a twist of $y^2=x^5-x$, as one may verify by computing their respective Igusa invariants.
As shown below, the Jacobian of $C^0_2$ is $\Q$-isogenous to the square of an elliptic curve defined over $\Q$, a property that the curve $y^2=x^5-x$ does not enjoy.
We also note that the minimal field of definition of the endomorphisms of the Jacobian of $C^0_2$ is $\Q(\sqrt{-2})$, but for $y^2=x^5-x$ it is $\Q(i,\sqrt{-2})$.

Let $E^0_2$ and $E^0_3$ denote the elliptic curves defined over $\Q$ by the equations
$$ E^0_2\colon Y^2=X^3-5X^2-5X+1\,,\qquad E^0_3\colon Y^2=X^3+1\,.$$
We note that $j(E^0_2)=2^65^3$ and $j(E^0_3)=0$, thus $E^0_2$ has CM by $\Q(\sqrt{-2})$ and $E^0_3$ has CM by $\Q(\sqrt{-3})$.

To simplify notation, throughout this section $d$ denotes either 2 or 3, and we write $C^0$ for $C^0_d$, $E^0$ for $E^0_d$, and $M$ for $\Q(\sqrt{-d})$.
We use $C$ to denote a twist of $C^0$ defined over $\Q$.
In the context of  \S\ref{section: squares}, we are specializing $A_\Qbar\sim E^2_\Qbar$ to the case where $A=\Jac(C)$ and $E=E^0$, as we now show.

\subsection{Fields of definition of isomorphisms}\label{section: fields definition}

\begin{lemma}\label{lemma:quotient}
$\Jac(C^0_d)$ is $\Q$-isogenous to $(E^0_d)^2$.
\end{lemma}

\begin{proof}
We proceed as in the proof of Lemma 4.1 in \cite{FL11}.
The quotient of $C_d^0$ by the non-hyperelliptic involution $\alpha(x,y)=(-x,y)$ is precisely the elliptic curve $E^0_d$, thus $\Jac(C_d^0)\sim_\Q E^0_d\times E$, where $E$ is also an elliptic curve defined over $\Q$.
The automorphism $\gamma(x,y)=(1/x,y/x^3)$ does not commute with $\alpha$, which implies that $\End(\Jac(C^0))$ is nonabelian, and therefore $\Jac(C_d^0)\sim_{\Q}(E^0_d)^2$.
\end{proof}

\begin{lemma}\label{lemma: automorphisms vs endomorphisms} The minimal number field over which all the automorphisms of $C_{\Qbar}$ are defined coincides with the minimal number field over which all the endomorphisms of $\Jac(C)_\Qbar$ are defined.
\end{lemma}

\begin{proof} Let $K_a$ (resp.\ $K_e$) denote the minimal number field over which all the automorphisms of $C_\Qbar$ (resp.\ all the endomorphisms of $\Jac(C)_\Qbar$) are defined.
The fact that $\Aut(C_{K_a})$ is nonabelian and contains a non-hyperelliptic involution implies that $\Jac(C)_{K_a}\sim E^2$, where $E$ is an elliptic curve defined over $K_a$.
Since $E$ has CM by $M$, it follows that $K_e=K_aM$. But \cite[Prop.\ 7.3.1]{Car01} asserts that $M=\Q(\sqrt{-3})$ is already contained in $K_a$ if $C$ is a twist of $C^0_3$, whereas \cite[Prop.\ 8]{Car06} states that $M=\Q(\sqrt{-2})$ is already contained in $K_a$ if $C$ is a twist of $C^0_2$.
\end{proof}

We use $K$ to denote the field given by Lemma \ref{lemma: automorphisms vs endomorphisms}.  We note that $K$ is a Galois extension of $\Q$, and we have $M\subseteq K$, with equality in the case $C=C^0$.

\begin{lemma}\label{lemma:equiv}
Let $\phi$ be an isomorphism from $C^0_\Qbar$ to $C_\Qbar$.  The following number fields coincide:
\begin{enumerate}[\rm (i)]
\item the minimal field over which all isomorphisms from $C^0_\Qbar$ to $C_\Qbar$ are defined;
\item the compositum of $K$ (or even just $M$) and the minimal field $L_\phi$ over which $\phi$ is defined;
\item the minimal field over which all homomorphisms from $\Jac(C^0)_\Qbar$ to $\Jac(C)_\Qbar$  are defined;
\item the minimal field over which all  homomorphisms from $E^0_\Qbar$ to $\Jac(C)_\Qbar$  are defined.
\end{enumerate}
\end{lemma}

\begin{proof} Let $L_{1}$, $L_{2}$, $L_{3}$, and $L_{4}$ denote the fields defined by (i), (ii), (iii), and (iv), respectively.
Any isomorphism $\psi$ from $C^0_\Qbar$ to $C_\Qbar$ can be written as $\psi=\alpha\circ\phi$ and $\phi\circ\alpha^0$ for some $\alpha\in\Aut(C_K)$ and some $\alpha^0\in\Aut(C^0_{M})$.
This implies that $L_{1}\subseteq ML_\phi \subseteq KL_\phi = L_{2}$.
Conversely, for any $\alpha^0\in\Aut(C^0_{M})$ and $\alpha\in\Aut(C_{K})$, the compositions $\alpha\circ\phi$ and $\phi\circ \alpha^0$ are isomorphisms from $C^0_\Qbar$ to $C_\Qbar$.
It follows that $L_{2}\subseteq L_{1}$.
Thus we have shown $L_1=ML_\phi=KL_\phi=L_2$.

The isomorphism from $C^0_{L_\phi}$ to $C_{L_\phi}$ induces an isogeny $\Jac(C^0)_{L_\phi}\sim\Jac(C)_{L_\phi}$, which we also denote $\phi$.
Any homomorphism from $\Jac(C^0)_\Qbar$ to $\Jac(C)_\Qbar$ can be written as $\psi\circ\phi$ for some $\psi\in\End(\Jac(C)_\Qbar)\otimes\Q$.
This implies that $L_{3}\subseteq L_{\phi}K, L_\phi M=L_{2}$.
Conversely, it is clear that $L_1$ is contained in $L_{3}$.

Any endomorphism $\phi$ from $\Jac(C^0)_\Qbar$ to $\Jac(C)_\Qbar$ can be written as $\phi_2\circ\phi_1$, where $\phi_1\in\Hom(\Jac(C^0),(E^0)^2)\otimes\Q$ and $\phi_2\in (\Hom(E^0_{L_{4}},\Jac(C)_{L_4})\otimes\Q)^2$. Thus $L_3\subseteq L_4$.
Conversely, any homomorphism from $E_{\Qbar}^0$ to $\Jac(C)_\Qbar$ can be written as $\phi_2\circ\phi_1$, where $\phi_1\in\Hom(E^0,\Jac(C^0))\otimes\Q$ and $\phi_2$ is an element of $\Hom(\Jac(C^0)_{L_{3}},\Jac(C)_{L_{3}})\otimes\Q\,.$  Thus $L_4\subseteq L_3$.
\end{proof}

We use $L$ to denote the field given by Lemma \ref{lemma:equiv}, and we note that $L$ is a Galois extension of $\Q$ that contains $K$.

\begin{remark}
If $A$ is the abelian three-fold $E^0\times \Jac(C)$, we observe that $L$ coincides with the minimal field over which all the endomorphisms of $A_\Qbar$ are defined.  It follows that the component group of $\ST(A)$ is isomorphic to $\Gal(L/\Q)$.
\end{remark}

\subsection{The Galois module \texorpdfstring{$\Hom(E^0_L,\Jac(C)_L)$}{}}\label{section: representations twists}

In this subsection we compute $\theta_{M,\sigma}(E^0,\Jac(C))$, strengthening Lemma \ref{lemma: trans coeff} in the case where $A\sim_\Q\Jac(C)$.
We take advantage of the following fact: the group $\Gal(L/\Q)$ is isomorphic to a subgroup of $G_{C^0}:=\Aut(C^0_M)\rtimes \Gal(M/\Q)$. Here the action of $\Gal(M/\Q)$ on $\Aut(C^0_M)$ is the natural one (see \cite[\S2]{FL11}).

More precisely, let $\phi\colon C_L\rightarrow C^0_L$ be an isomorphism.
Then
$$\lambda_\phi\colon \Gal(L/\Q)\hookrightarrow G_{C^0}\,\qquad \lambda_{\phi}(\sigma)=(\phi({}^\sigma\phi)^{-1},\pi_{L/M}(\sigma))\,$$
is a monomorphism of groups, where $\pi_{L/M}\colon\Gal(L/\Q)\rightarrow \Gal(M/\Q)$ is the natural projection, as in \cite[Lem.\ 2.1]{FL11}.
Now let
$$\Res^\Q_M\lambda_{\phi}\colon\Gal(L/M)\hookrightarrow \Aut(C^0_M)$$
be the restriction of $\lambda_\phi$ at $\Gal(L/M)$. Consider the $2$-dimensional $M$-rational representation
$$\theta_{E^0,C^0}\colon\Aut(C^0_M)\rightarrow\Aut_\Qbar( \Hom(E^0_M,\Jac(C^0)_M)\otimes_{M,\sigma}\Qbar)$$
defined by $\theta_{E^0,C^0}(\alpha)(\psi)=\alpha\circ\psi$.
As in \cite[Thm.\ 2.1]{FL11}, one then has
\begin{equation}\label{equation: iso theta}
\theta_{E^0,C^0}\circ\Res^\Q_M\lambda_\phi\simeq \theta_{M,\sigma}(E^0,\Jac(C))\,,
\end{equation}
where $\theta_{M,\sigma}(E^0,\Jac(C))$ is the representation of $\Gal(L/M)$ in Definition~\ref{definition: thetaMsigma}.

\begin{lemma}\label{lemma: twisting rep} Let $C$ be a twist of $C^0$. Then:
$$
\Tr \theta_{E^0,C^0}=\begin{cases}
\chi_4 \text{ or } \chi_5 &\text{ if $C^0=C^0_2$ (see Table \ref{table: charactertable2}).}\\
\chi_8 \text{ or } \chi_9 &\text{ if $C^0=C^0_3$ (see Table \ref{table: charactertable3}),}
\end{cases}
$$
\end{lemma}

\begin{proof} A glance at the character table of $\Aut(C^0_M)$ (see Tables \ref{table: charactertable2} and  \ref{table: charactertable3} in \S\ref{section: tables}) tells us that any $M$-rational faithful representation of degree $2$ 
must have trace $\chi_4$ or $\chi_5$ when $C^0=C^0_2$, or trace $\chi_8$ or $\chi_9$ when $C^0=C^0_3$. The two possibilities in each case correspond to the two different embeddings of $M$ into $\Qbar$.
\end{proof}

\begin{proposition}\label{proposition: degree}
The index of $K$ in $L$ is at most $2$.
\end{proposition}

\begin{proof} As in Lemma \ref{lemma:quotient}, let $\alpha$ be the non-hyperelliptic involution $\alpha(x,y)=(-x,y)$ of $C^0$.
Let $E$ be the elliptic curve $C_K/\langle \phi^{-1}\alpha\phi\rangle$ defined over $K$ (note that $ \phi^{-1}\alpha\phi$ is an automorphism of $C$, all of which are defined over $K$).
The isomorphism $\phi\colon C_L\rightarrow C^0_L$ induces an isomorphism $\tilde\phi\colon E_L \rightarrow E^0_L$. Thus $E$ is a $K$-twist of $E^0$.
From characterization (iii) of $L$ in Lemma \ref{lemma:equiv}, it is clear that $L$ is the compositum of $K$ and the minimal field $L_{\tilde\phi}$ over which $\tilde\phi$ is defined.

When $C^0=C^0_2$, we have $j(E)\ne 0, 1728$, and by \cite[p.\ 304]{Sil09} it follows that $\tilde\phi$ is then defined over a quadratic extension of $K$ and $[L:K]\le 2$.
When $C^0=C^0_3$, we have $j(E)=0$, and in this case $L=K(\sqrt[6]{\gamma})$, for some $\gamma\in K$.
Let $L_0=K(\sqrt\gamma)$.  It suffices to show that $\sqrt[3]\gamma\in L_0$.

Suppose for the sake of contradiction that $\sqrt[3]\gamma\not\in L_0$.
Then $\Gal(L/L_0)\simeq \cyc 3$.
Lemma \ref{lemma: twisting rep} then implies that if $\tau$ is a nontrivial element of $\Gal(L/L_0)$, then $\Tr\theta_{M,\sigma}(E^0,\Jac(C))(\tau)=-1$.
Therefore, the restriction of the representation afforded by the $\Gal(L/M)$-module $\Hom(E^0_L,\Jac(C)_L)\otimes_{M,\sigma}\Qbar$ to $\Gal(L/L_0)$ is
$$\Res^M_{L_0}\theta_{M,\sigma}(E^0,\Jac(C))\simeq \chi \oplus \overline\chi\,,$$
where $\chi$ is any of the two nontrivial characters of $\Gal(L/L_0)$.
As in \cite[Thm.\ 3.1]{Fit10}, we have
$$\Res_{L_0}^M\theta_{M,\sigma}(E^0,\Jac(C))\otimes V_\sigma(E^0)\simeq V_\sigma(\Jac(C)),$$
as $\Qbar_\ell[G_{L_0}]$-modules.
This implies that
\begin{equation}\label{equation: dec twist1}
V_\sigma(\Jac(C))\simeq \bigl(\chi\otimes V_\sigma(E^0)\bigr)\oplus \bigl(\overline\chi \otimes V_\sigma(E^0)\bigr)\,,
\end{equation}
as $\Qbar_\ell[G_{L_0}]$-modules.
However, as seen in Lemma \ref{lemma: automorphisms vs endomorphisms}, $\Jac(C)_{L_0}\sim E_{L_0}^2$, which implies the following isomorphism of $\Qbar_\ell[G_{L_0}]$-modules:
\begin{equation}\label{equation: dec twist 2}
V_\sigma(\Jac(C))\simeq V_\sigma(E)^{2\oplus}\,.
\end{equation}
But now (\ref{equation: dec twist1}) and (\ref{equation: dec twist 2}) together imply $V_\sigma(E^0)\simeq \chi\otimes V_\sigma(E^0)$, which is impossible.
(We remark that if  $\Res^M_{L_0}\theta_{M,\sigma}(E^0,\Jac(C))\simeq \chi ^{2\oplus}$, one does not reach a contradiction; see Example \ref{example: counterexamples}).
\end{proof}

\begin{proposition}\label{proposition: hyperprop} Let $w$ be the hyperelliptic involution of $C^0$. Then $[L\colon K]=2$ if and only if $(w,\id)\in G_{C^0}$ lies in the image of $\lambda_\phi$. If $[L\colon K]=2$, then the preimage of $(w,\id)$ by $\lambda_\phi$ is the nontrivial element $\omega$ of $\Gal(L/K)$.
\end{proposition}

\begin{proof}
We first suppose that $[L\colon K]=2$.
Observe that for both $C^0_2$ and $C^0_3$, if $\alpha\in \Aut(C^0_M)$ and  $\Tr \theta_{E^0,C^0}(\alpha)=-2$, then $\alpha=w$.
In view of the isomorphism in (\ref{equation: iso theta}), it thus suffices to prove that $\theta_{M,\sigma}(E^0,\Jac(C))(\omega)=-2$.
From the proof of Proposition \ref{proposition: degree}, we know that $\Jac(C)_K\sim E^2$, where $E$ is an elliptic curve defined over $K$ with CM by $M$.
Fix an isomorphism $\psi_1\colon E^0_L\rightarrow  E_L$.
Fix an isogeny $\psi_2\colon E_K\times E_K \rightarrow \Jac(C)_K$.
For $i=1,2$, let $\iota_i\colon E_K\rightarrow E_K\times E_K$ denote the natural injection to the $i$th factor. Then $\psi_2\circ\iota_1\circ\psi_1$ and $\psi_2\circ\iota_2\circ\psi_1$ constitute a basis of the $\Qbar[\Gal(L/M)]$-module $\Hom(E^0_L,\Jac(C)_L)\otimes_{M,\sigma}\Qbar$.
The claim follows from the fact that ${}^{\omega}\psi_1=-\psi_1$, ${}^{\omega}\psi_2=\psi_2$, and ${}^{\omega}\iota_i=\iota_i$.

Now suppose that $[L\colon K]=1$. Recall the monomorphism
$$\overline\lambda_\phi\colon\Gal(K/\Q)\hookrightarrow \bigslant{\Aut(C^0_M)}{\langle w\rangle}\rtimes \Gal(M/\Q)$$
of equation (\ref{remark: injection}).  The commutativity of the diagram
$$
\xymatrix{
&G_{C^0}\ar[rd] & \\
\Gal(L/\Q)\ar@{=}[r]\ar@{^{(}->}[ur]^{\lambda_\phi} & \Gal(K/\Q) \ar@{^{(}->}[r]^{\overline\lambda_\phi}& \bigslant{G_{C^0}}{\langle (w,\id)\rangle}\,.}
$$
implies that $(w,\id)$ does not lie in the image of $\lambda_\phi$.
\end{proof}

\begin{remark}
Let $H_0:=\lambda_\phi(\Gal(L/M))$.
 If $(\id,\tau)$ lies in the image of $\lambda_\phi$, then $\lambda_\phi (\Gal(L/\Q))= H_0\rtimes \langle (\id,\tau)\rangle$; indeed, $H_0$ is normal in $\lambda_\phi (\Gal(L/\Q))$, since its index is $2$, and  $H_0\cap \langle (\id,\tau)\rangle$ is trivial.
In this case, $H_0$ is stable under the action of $\Gal(M/\Q)$. However it is not true in general that $\Gal(L/\Q)\simeq H_0\rtimes \langle (\id,\tau)\rangle$ or that $H_0$ is stable under the action of $\Gal(M/\Q)$.
\end{remark}

\begin{proposition}\label{proposition: trthetaMEA}  For $\tau$ in $\Gal(L/\Q)$, let $s=s(\tau)$, $r=r(\tau)$, and $t=t(\tau)$ denote the orders of $\tau$, the projection of~$\tau$ on $\Gal(K/\Q)$, and the projection of~$\tau$ on $\Gal(M/\Q)$, respectively.  The following hold:
\begin{enumerate}[\rm (i)]
\item The triple $(s,r,t)$ is one of the $13$ triples listed in Table  \ref{table:trthetaMEA}.
\item If $\tau$ fixes $M$, then the triple $(s,r,1)$ determines, up to sign, the quantities
\begin{align}
a_1(\theta)(\tau)&=\Tr\theta_{M,\sigma}(E^0,\Jac(C))(\tau),\\
a_2(\theta)(\tau)&=\det\theta_{M,\sigma}(E^0,\Jac(C))(\tau),
\end{align}
as specified in Table \ref{table:trthetaMEA}.
\item For each triple $(s,r,t)$, let $F_{(s,r,t)}\colon [-2,2]\rightarrow [-4,4]\times [-2,6]$ be the map defined in Table \ref{table:trthetaMEA}.
For every prime $p>3$ unramified in $L$ of good reduction for both $\Jac(C)$ and $E^0$, there exists a unique triple $(s,r,t)$ such that
\begin{equation}\label{equation: functions}
F_{(s,r,t)}(a_1(E^0)(p))=(u\cdot a_1(\Jac(C))( p), a_2(\Jac(C))( p)),
\end{equation}
with $u=\pm 1$ $($in fact, $u=1$ for $(s,r,t)\not = (6,6,1)$ and $(8,4,1)$$)$. Moreover, the unique triple $(s,r,t)$ for which $(\ref{equation: functions})$ holds is $\bigl(f_L(p),f_K(p),f_M(p)\bigr)$, where $f_F(p)$ is the residue degree of $p$ in $F$.
\end{enumerate}

\begin{remark}
Observe that for a prime $p$ unramified in $L$ such that $f_M(p)=1$, we have
 $$a_2(\Jac(C))( p)=a_2(\theta)(\Frob_ p)\cdot a_1(E^0)(p)^2+|a_1(\theta)(\Frob_{p})|^2-2 a_2(\theta)(\Frob_ p)\,,$$
where $a_2(\theta)(\Frob_ p)=\pm 1$.
It follows that for any two twists $C$ and $C'$ of $C^0$ we have
$$\hat a_2(\Jac(C))( p) \equiv \pm\hat a_2(\Jac(C'))( p)\quad\pmod{ p}\,.$$
where $\hat a_i(A)(p)=p^{i/2}a_i(A)(p)$ is the integer that appears as the coefficient of $T^i$ in the (unnormalized) $L$-polynomial $L_p(A,T)$.
\end{remark}

\begin{table}
\begin{center}
\caption{The triples for $(s,r,t)$ associated to $\tau\in\Gal(L/\Q)$ (as defined in Proposition~\ref{proposition: trthetaMEA}), and corresponding values of $F_{(s,r,t)}(x)$, $a_1(\theta)(\tau)$, and $a_2(\theta)(\tau)$.}
\label{table:trthetaMEA}
\vspace{6pt}
\begin{tabular}{rrrr}
$(s,r,t)$ & $F_{(s,r,t)}(x)$ & $a_1(\theta)(\tau)$ & $a_2(\theta)(\tau)$ \\\hline\vspace{-10pt}
&&&\\
$(1,1,1)$ & $(2x,x^2+2)$                 & $2$                & $1$ \\
$(2,1,1)$ &$(-2x,x^2+2)$                & $-2$               & $1$  \\
$(2,2,1)$ &$(0,-x^2+2)$                 & $0$                & $-1$  \\
$(3,3,1)$ & $(-x,x^2-1)$                & $-1$               & $1$   \\
$(4,2,1)$ & $(0,x^2-2)$                 & $0$                & $1$   \\
$(6,3,1)$ &  $(x,x^2-1)$                 & $1$                & $1$ \\
$(6,6,1)$ & $(\sqrt{3(4-x^2)},-x^2+5)$                 & $\pm\sqrt{-3}$     & $-1$ \\
$(8,4,1)$ & $(\sqrt{2(4-x^2)},-x^2+4)$                 & $\pm \sqrt{-2}$    & $-1$  \\
$(2,2,2)$ & $(0,2)$ &  $-$ & $-$ \\
$(4,2,2)$ & $(0,-2)$&  $-$ & $-$ \\
$(6,6,2)$ & $(0,-1)$&  $-$ & $-$ \\
$(8,4,2)$ & $(0,0)$&  $-$ & $-$ \\
$(12,6,2)$& $(0,1)$&  $-$ & $-$ \\\hline
\end{tabular}
\end{center}
\end{table}
\end{proposition}

\begin{proof} For assertion (i) assume first that $t=1$.
Observe that $s$ is the order of $\lambda_\phi(\tau)$ in $\Aut(C^0_M)$, and $r$ is the order of the projection of $\lambda_\phi(\tau)$ in $\bigslant{\Aut(C^0_M)}{\langle w\rangle}$.
Let $c$ denote the conjugacy class of $\lambda_\phi(\tau)$ in $\Aut(C^0)$.
One finds that the pairs $(r,s)$ are determined by the conjugacy class of $\tau$ as follows:
\begin{center}
\begin{tabular}{ll}
if $C^0=C^0_3$ &
$\left\{\begin{array}{lllllllll}
c\colon & 1a & 2a & 2b,2c & 3a & 4a & 6a, 6b & 6c\\
(r,s)\colon & (1,1) & (2,1)  & (2,2) & (3,3) & (4,2) & (6,6) & (6,3)
\end{array}\right.$
\\[8pt]
if $C^0=C^0_2$
&
$\left\{\begin{array}{lllllllll}
c\colon & 1a & 2a & 2b & 3a & 4a & 6a & 8a\\
(r,s)\colon & (1,1) & (2,1)  & (2,2) & (3,3) & (4,2) & (6,3) & (8,4)
\end{array}\right.$
\end{tabular}
\end{center}
(see Tables \ref{table: charactertable3} and \ref{table: charactertable2} for the names of the conjugacy classes).
Assertion (ii) now follows immediately by applying the isomorphism in (\ref{equation: iso theta}) and Lemma \ref{lemma: twisting rep}.
If $t=2$, then $r$ must be $2$, $4$, or $6$, and the fact that $[L\colon K]\leq 2$ limits $(s,r,t)$ to either one of the last $5$ triples in Table~\ref{table:trthetaMEA}, or $(4,4,2)$.
The latter possibility is ruled out by Proposition \ref{proposition:trthetaQ}: if $s=4$, then for every prime $p$ for which $\Frob_p$ lies in the same conjugacy class of $\tau$ in $\Gal(L/\Q)$, we have $\overline L_p(\Jac(C),T)=(1-T^2)^2$, and the only quotients of roots of this polynomial are $1$ and $-1$.
This implies that $\theta_M(\Jac(C))(\tau)$ has order 2, and since $\theta_M(\Jac(C))$ is faithful we must have $r=r(\tau)=2$, not 4.

For $t=1$, the existence statement in (iii) follows from combining Lemma \ref{lemma: trans coeff} with statement (ii), and
for $t=2$, it follows from the proof of Corollary~\ref{corollary: a2 moments}.
The uniqueness of the map $F_{(s,r,t)}$ satisfying (\ref{equation: functions}) at a prime $p > 3$ may be verified by noting that the graphs of the 13 functions $F_{(s,r,t)}$ intersect in only finitely many points in $\R^3$, none of which corresponds to a possible value of $(a_1(E^0)(p),a_1(\Jac(C))(p),a_2(\Jac(C))(p))$ for any prime $p > 3$.
Finally, we note that if $\tau=\Frob_p$ then $(s(\tau),r(\tau),t(\tau))=(f_L(p),f_K(p),f_M(p))$.
\end{proof}


We now give two examples of abelian varieties $A$ such that $A_\Qbar \sim (E^0_\Qbar)^2$ for which the conclusions of Propositions \ref{proposition: degree} and \ref{proposition: trthetaMEA} do not hold because $A$ is not $\Q$-isogenous to the Jacobian of a twist of $C^0$.
In the two examples below we use the elliptic curve
$$\tilde{E}^0_3\colon y^2=x^3+2$$
defined over $\Q$, which is a twist of $E^0_3$.

\begin{example}
Let $A=E^0_3\times \tilde{E}^0_3$.  Then $K=L=\Q(\sqrt[6]{2},\zeta_3)$.
If $\tau\in\Gal(L/M)$ and $s(\tau)=3$, then $a_1(\theta)(\tau)=1+\zeta_3$ or $1+\overline\zeta_3$ and $a_2(\theta)(\tau)=\zeta_3$ or $\overline\zeta_3$, which do not lie in $\Q$.
Thus by Proposition \ref{proposition: trthetaMEA}, $A$ is not $\Q$-isogenous to the Jacobian of any $\Q$-twist  of $C^0_3$.
Moreover, for $p=7$, one can compute
that $\hat a_2(A)(7)=30$, while $\hat a_2((E^0_3)^2)(7)=\hat a_2(\Jac(C^0_3))(7)=18$. Thus
$$\hat a_2(A))( 7) \not\equiv \pm\hat a_2(\Jac(C^0))( 7)\quad\pmod{ 7}\,.$$
\end{example}

\begin{example}\label{example: counterexamples}
Let $A=(\tilde{E}^0_3)^2$, then $L=\Q(\sqrt[6]2,\zeta_3)$ and $K=\Q(\zeta_3)$.
Then $[L\colon K]=6,$, and, by Proposition \ref{proposition: degree}, $A$ is not $\Q$-isogenous the Jacobian of any $\Q$-twist  of $C^0_3$.
In the context of the proof of Proposition \ref{proposition: degree}, $L_0=\Q(\sqrt2,\zeta_3)$ and $\Res^\Q_{L_0}\theta(E^0,A)\simeq \chi ^{2\oplus}$, rather than $\Res^\Q_{L_0}\theta(E^0,A)\simeq \chi \oplus\overline\chi$, which avoids the contradiction used in the proof.
Moreover, for $A$ we may have $s(\tau)=3$ and $r(\tau)=1$, which gives a pair $(s,r)$ that can not occur for the Jacobian of any $\Q$-twist of $C^0_3$, by part $(ii)$ of Proposition \ref{proposition: trthetaMEA}.
\end{example}

\subsection{The triples  \texorpdfstring{$T(C)$}{}}\label{section: triples of Galois groups}

In this section we will determine the possible values of the triple $T(C)$, which denotes the isomorphism class $[\Gal(L/\Q), \Gal(K/\Q), \Gal(L/M)]$.
To specify triples explicitly, we use identifiers from the Small Groups Library \cite{SGL} found in computer algebra systems such as GAP \cite{GAP4} and Magma \cite{Magma}.
These identifiers consist of a pair of positive integers $\langle n, m\rangle$, where $n$ is the order of the group and~$m$ distinguishes the group from other groups of order $n$ but otherwise has no meaning.  We also recall from $\S$\ref{section: representations twists} the embeddings
\begin{center}
\begin{tabular}{lll}
$\lambda_\phi\colon \Gal(L/\Q)\hookrightarrow G_{C^0}$, & & $\Res\lambda_\phi\colon\Gal(L/M)\hookrightarrow \Aut(C^0_M)$,\\
$\overline\lambda_\phi\colon \Gal(K/\Q)\hookrightarrow \bigslant{G_{C^0}}{\langle (w,1)\rangle}$, &  & $\Res\overline \lambda_\phi\colon\Gal(K/M)\hookrightarrow \bigslant{\Aut(C^0_M)}{\langle w\rangle}$,
\end{tabular}
\end{center}
where $w$ denotes the hyperelliptic involution of $C^0$.

\begin{lemma}\label{lemma: specify groups}
The groups $G_{C^0}$, $G_{C^0}/\langle (w,1) \rangle$, $\Aut(C^0_{M})$,  and $\Aut(C^0_{M})/ \langle w \rangle$  are as follows:
\begin{center}
\begin{tabular}{lllll}
$C^0$ &  $G_{C^0}$ & $G_{C^0}/\langle (w,1) \rangle$  & $\Aut(C^0_{M})$ & $\Aut(C^0_{M})/ \langle w \rangle$\vspace{2pt}\\\hline\vspace{-10pt}
&&&&\\
$C^0_3$ &  $\langle 48, 38\rangle$ & $\langle 24, 14\rangle$  & $\langle 24, 8\rangle$ & $\langle 12, 4\rangle$ \vspace{2pt}\\
$C^0_2$ & $\langle 96,193\rangle$ & $\langle 48,48\rangle$   & $\langle 48,29\rangle$ & $\langle 24,12\rangle$
\end{tabular}
\end{center}
\end{lemma}

\begin{proof} We show how to compute $G_{C^0}$  and $\Aut(C^0_{M})$; the respective quotients are then easily obtained. Recall that if $C/\Q$ is a genus 2 curve, given by a hyperelliptic equation  $y^2=f(x)$, where $f(x)\in \Q[x]$, then for any $\alpha\in\Aut(C_\Qbar)$ there exist $m,\,n\,,p\,,q\in \Qbar$ such that
\begin{equation}\label{equation: injection}
\alpha(x,y)=\left(\frac{mx+n}{px+q},\frac{mq-np}{(px+q)^3}y\right);
\end{equation}
see, for example, \cite{Car06}.  Let
$$
\iota(\alpha):=\begin{pmatrix}
  m & n \\
  p & q
\end{pmatrix}.$$
The map
$\iota\colon \Aut(C_{\Qbar})\rightarrow \GL_2(\Qbar)$ that sends $\alpha$ to $\iota(\alpha)$
is a $G_\Q$-equivariant monomorphism. For $d=2,3$, we have $\Aut((C^0_d)_M)=\langle U_d,\, V_d\rangle$, where
\begin{center}
\begin{tabular}{lll}
$
U_3=\begin{pmatrix}
0 & 1\\
1 & 0
\end{pmatrix}$, &  &
$V_3=\frac{1}{2}
\begin{pmatrix}
0 & -1+\sqrt{-3} \\
1+\sqrt{-3} & 0\\
\end{pmatrix},$\\[10pt]
$U_2=\frac{1}{2}
\begin{pmatrix}
\sqrt{-2}-1 & 1 \\
 1 & 1+\sqrt{-2}
\end{pmatrix},$ &  &
$V_2=\frac{1}{2}
\begin{pmatrix}
1 & -\sqrt{-2}+1\\
 -1-\sqrt{-2} & 1\\
\end{pmatrix}.$\\
\end{tabular}
\end{center}
One can readily check that $U_d$ and $V_d$ represent automorphisms of $(C^0_d)_M$, and that they generate a group of order 48 if $d=2$ and of order 24 if $d=3$,
which are known to be the orders of $\Aut((C_d^0)_M)$ .
With this explicit representation, the isomorphism type of $\langle U_d, V_d\rangle$ is then easily determined by a computer algebra system.
The group $G_{C_d^0}$ is then determined by explicitly computing the semidirect product $\langle U_d,\, V_d\rangle\rtimes \Gal(\Q(\sqrt{-d})/\Q)$.
\end{proof}

Let $\tilde T(C)$ denote the triple
$$(\lambda_\phi(\Gal(L/\Q)), \lambda_\phi(\Gal(L/M)),\lambda_\phi(\Gal(L/M)))$$
in $G_{C^0}\times G_{C^0}\times G_{C^0}$. Since $\lambda_\phi$ is injective, the conjugacy class of $\tilde T(C)$ determines $T(C)$ and $z(C)$, where $z(C)$ is the vector in Definition~\ref{definition: zvector}.
In order to bound the number of possibilities for $T(C)$ and $z(C)$, we first bound the number of possible triples $\tilde T(C)$, up to conjugation.

\begin{lemma}\label{lemma: primera}
Let $H$, $N$ and $H_0$ be subgroups of $G_{C^0}$. If $\tilde T(C)=(H,N,H_0),$ then the following conditions must be satisfied:

\begin{enumerate}[\rm (i)]
\item $H_0$ and $H\cap\Aut(C^0_M)\times \langle 1\rangle$ coincide and have order $|H|/2$.
\item $N$ and $\langle (w,1)\rangle\cap H_0$ coincide.
\end{enumerate}
\end{lemma}

\begin{proof} Let $\Gal(M/\Q)=\{1,\tau\}$. Then
\begin{center}
\begin{tabular}{l}
$H_0=\lambda_\phi(\Gal(L/M))\subseteq (\Aut(C^0_M)\times\{ 1\}) \cap H,$\\[4pt]
$H_1:=\lambda_\phi(\Gal(L/\Q)\setminus\Gal(L/M))\subseteq (\Aut(C^0_M)\times\{ \tau\}) \cap H.$
\end{tabular}
\end{center}
The injectivity of $\lambda_\phi$ implies that $|H_0|=|H_1|=|H|/2$, and (i) follows from the fact that $H=H_0\sqcup H_1$.

Proving (ii) is equivalent to showing that $(w,1)$ lies in the image of $\lambda_\phi$ if and only if $[L\colon K]=2$.  But this has already been proved, see Proposition \ref{proposition: hyperprop}.
\end{proof}

\begin{proposition}\label{proposition: candidates} Let $H$, $N$, and $H_0$ be subgroups of $G_{C^0}$, that satisfy conditions $(i)$ and $(ii)$ of Lemma \ref{lemma: primera}. Then the following hold:
\begin{enumerate}[\rm (i)]
\item For $C^0=C^0_2$ (resp.\ $C^0=C^0_3$), there are exactly $27$ (resp.\ $38$) possibilities for the conjugacy class of $(H,N,H_0)$ in $G_{C^0}\times G_{C^0}\times G_{C^0}$.

\item For $C^0=C^0_2$ (resp.\ $C^0=C^0_3$), the $27$ (resp.\ $38$) possibilities for the conjugacy class of $(H,N,H_0)$
give rise to exactly the $23$ (resp.\ $23$)  isomorphism classes $[H,H/N,H_0]$ and vectors $z(H,N,H_0)$ listed in Table~\ref{Table: candidates2} (resp.\ Table~\ref{Table: candidates3}).
Moreover, $[H,H/N,H_0]$ and $z(H,N,H_0)$ determine each other uniquely.

\item  For $C^0=C^0_2$ (resp.\ $C^0=C^0_3$) the triple $T(C)$ and the vector $z(C)$ must be among those listed in Table~\ref{Table: candidates2} (resp.\ Table~\ref{Table: candidates3}),
and $T(C)$ and $z(C)$ determine each other uniquely
\end{enumerate}
\end{proposition}

\begin{proof} For (i), recall that $G_{C^0_3}\simeq \langle 48, 38\rangle$ and $\Aut((C^0_3)_M)\simeq \langle 24,8\rangle$. The following three facts permit us to work with $G_{C^0_3}$ and $\Aut((C^0_3)_M)$ as abstract groups. First, there are exactly two subgroups $A_1$ and $A_2$ of $\langle 48, 38\rangle$ isomorphic to $\langle 24,8\rangle$. Second, there is a unique nontrivial central involution $\hat{w}$ in $\langle 48, 38\rangle$, and it lies in both $A_1$ and $A_2$. Third, consider the two lists of triples of groups, up to conjugation,
$$\mathcal L_i=\bigslant{\{(H,\langle \hat{w}\rangle\cap H,H\cap A_i) \,|\,H\subseteq\langle 48, 38\rangle,\,|H\cap A_i|=|H|/2\}}{\sim}\,,\quad i=1,\,2,\,$$
where $(H,\langle \hat{w}\rangle\cap H,H\cap A_i)\sim (H',\langle \hat{w}\rangle\cap H',H'\cap A_i)$ if $H$ and $H'$ are conjugated in $\langle 48, 38\rangle$. Then the lists $\mathcal L_1$ and $\mathcal L_2$ coincide; write $\mathcal L$ for this list. For $C^0=C^0_2$, the three previous facts can be checked to hold \emph{verbatim} when replacing $\langle 48,38 \rangle$ and $\langle 24,8\rangle$ by $\langle 96, 193\rangle$ and $\langle 48, 29\rangle$, respectively. For $C^0=C^0_3$, $\mathcal L$ has 38 elements and, for $C^0=C^0_2$, it has $27$ elements.

For (ii), for each of $C^0_2$ and $C^0_3$, we enumerate the triples $(H,N,H_0)$ in $\mathcal L$ and explicitly compute $[H,H/N,H_0]$ and $z(H,N,H_0)$ in each case  using a computer algebra system \cite{Magma}, obtaining the values listed in Tables~\ref{Table: candidates2} and \ref{Table: candidates3}.
 One then checks that $[H,H/N,H_0]=[H',H'/N',H_0']$ if and only if $z(H,H/N,H_0)=z(H',H'/N',H_0')$.

Statement (iii) follows immediately from (ii) and Lemma~\ref{lemma: primera}.
\end{proof}


\begin{proposition}\label{proposition: det ST}
The vector $z(C)$ and the triple $T(C)$ both uniquely determine the Sato-Tate group $\ST(\Jac(C))$.
\end{proposition}

\begin{proof} The 18 Sato-Tate groups $G$ that can occur over $\Q$ with $G^0\simeq\Unitary(1)$ (see \cite[Theorem 4.3]{FKRS}) are uniquely determined by the combination of:
\begin{enumerate}[(a)]
\setlength\itemindent{20pt}
\item the isomorphism classes of the groups $G/G^0$ and $G^{ns}/G^{ns,0}$, 
\item the vector $z_2(G)=(z_{2,2}(G),z_{2,-2}(G),z_{2,-1}(G),z_{2,0}(G),z_{2,1}(G))$,\footnote{Following the notation of \cite{FKRS}, recall that $z_{2,i}(G)$ denotes the number of connected components of $G$ all of whose elements have a constant characteristic polynomial, for which the coefficient of the quadratic term is equal to $i$. Note that the components of the vector $z_2(G)$ have been permuted with respect to the definition of $z_2(G)$ given in \cite{FKRS}.}
\end{enumerate}
where $G^{ns}$ is the index  2 subgroup of $G$ obtained by removing from $G$ those components all whose elements have a constant characteristic polynomial.

On the one hand, the isomorphism classes of the groups $G/G^0$ and $G^{ns}/G^{ns,0}$ are determined by $T(C)$, since $G/G^0\simeq \Gal(K/\Q)$ and $G^{ns}/G^{ns,0}\simeq\Gal(K/M)$. On the other hand, $z_2(G)$ is determined by $z(C)$; indeed, it follows from the construction of the Sato-Tate group in terms of the image of the $\ell$-adic representation attached to $\Jac(C)$ and from assertion (iii) of Proposition \ref{proposition: trthetaMEA}, that $z_2(G)\cdot[L\colon K]=z_2(C)$.
\end{proof}

\begin{corollary}\label{corollary: num g(C)} For each triple $[H,H/N,H_0]$ in Tables \ref{Table: candidates2} or \ref{Table: candidates3}, there exists a twist $C$ of $C^0$ such that $T(C)=[H,H/N,H_0]$ if and only if the corresponding row in the table is not marked with an asterisk.
Thus, for $C^0=C^0_2$ (resp.\ for $C^0=C^0_3$) there are exactly $20$ (resp.\ $21$) possibilities for $T(C)$.
\end{corollary}

\begin{proof} Observe that the triples marked with an asterisk in Tables \ref{Table: candidates2} and \ref{Table: candidates3} correspond to Sato-Tate groups (equivalently, Galois types) that cannot arise for abelian surfaces defined over $\Q$ (see  \cite[proposition 4.11]{FKRS}). For each of the triples $[H,H/N,H_0]$ that are not marked with an asterisk, a curve $C$ with $T(C)=[H,H/N,H_0]$ is exhibited in Tables \ref{Table: curves2} and \ref{Table: curves3} (for details on how the curves have been found, see \S\ref{section: search in a box}; for details on how $T(C)$ is computed for each of the curves see \S\ref{section: computationKL}.)
\end{proof}

\begin{remark} Observe that if the triple $[H,H,H_0]$ appears in Table \ref{Table: candidates2} or \ref{Table: candidates3}, then the triple $[H\times \cyc 2, H,H_0\times \cyc 2]$ also appears. In other words, if there exists a twist $C$ of $C^0$ such that $\Gal(L/\Q)=\Gal(K/\Q)$, then there exists a twist $C'$ of $C^0$ such that $\Gal(K'/\Q)=\Gal(K/\Q)$ and $\Gal(L'/\Q)\simeq \Gal(K/\Q)\times \cyc 2$. Here $K'$ (resp.\ $L'$) is the minimal field over which all the automorphisms of $C'$ (resp.\ all the isomorphisms between $C'$ and $C^0$) are defined. Indeed, if $C$ is given by the hyperelliptic equation $y^2=f(x)$, let $C'$ be the curve given by $dy^2=f(x)$, where $d\in\Q^*$ is not a square in $K$. We will use this remark in $\S$\ref{section: computations} for the computation of some of the curves.
\end{remark}

\begin{remark}
Among the $18$ Sato-Tate groups with identity component $\Unitary(1)$ that can occur over $\Q$, there are $13$ that are subgroups of $J(O)$ and $11$ that are subgroups of $J(D_6)$ ($6$ are subgroups of both). From Table \ref{Table: curves3} we see that the $13$ that are subgroups of $J(O)$ can all occur as $\Q$-twists of $C^0_2$, and the $11$ that are subgroups of $J(D_6)$ can all occur as $\Q$-twists of $C^0_3$.
\end{remark}

\section{Numerical computations}\label{section: computations}

We now describe the methods used to obtain the example curves $C$ listed in Tables \ref{Table: curves2} and \ref{Table: curves3}.
As in \S\ref{section: curves}, each curve $C$ is a $\Q$-twist of $C^0=C_d^0$, for $d=2,3$, where $\Jac(C_d^0)\sim (E_d^0)^2$ and $E_d^0$ is an elliptic curve with CM by $M=\Q(\sqrt{-d})$.
For $d=2$ we list 20 curves $C$ that are $\Q$-twists of the curve $C_2^0$ defined by $y^2=x^6-5x^4-5x^2+1$, realizing every possible triple $T(C)=[\Gal(L/\Q),\Gal(K/\Q),\Gal(L/M)]$ that can occur when $C$ is a $\Q$-twist of~$C_2^0$.
Recall that the fields $K$ and $L$ are the minimal fields of definition $\End(\Jac(C)_\Qbar)$ and $\Hom(\Jac(C)_\Qbar,E_\Qbar)$, respectively, as in Definition~\ref{definition: triple}.
Similarly, for $d=3$ we list 21 curves $C$ that are twists of the curve $C_3^0$ defined by $y^2=x^6+1$, realizing every possible triple $T(C)$ that can occur when $C$ is a $\Q$-twist of $C_3^0$.

For each of the two curves $C^0$ we followed the procedure outlined below:

\begin{enumerate}
\item
Generate a large set $S$ of $\Q$-twists of $C^0$.
\item
For each $C\in S$ compute a provisional value of the triple $T(C)$.
\item
Select a single representative $C$ for each distinct triple $T(C)$ and then verify the provisional value of $T(C)$ by explicitly computing the fields $K$ and $L$ and the triple $T(C)=[\Gal(L/\Q),\Gal(K/\Q),\Gal(L/M)]$.
\end{enumerate}
The purpose of the ``provisional" computation of $T(C)$ in step 2 is to avoid computing the fields $K$ and $L$ for all of the curves in $S$, which would have been infeasible.
Explicit computation of the fields $K$ and $L$ (and their Galois groups) for even a single curve $C$ can be quite time consuming, taking many hours or even days of computer time, and the sets $S$ that we used contained tens of thousands of curves.

In the rest of this section we fill in some of the details of the three steps listed above.

\subsection{Generating twists of \texorpdfstring{$C^0$}{}}\label{section: search in a box}
Explicit parameterizations of the families of twists of $C^0_3$ and $C^0_2$ are given by Cardona in \cite{Car01} and \cite{Car06}.
One can easily obtain a large set $S$ using these parameterizations.
However, the resulting curves tend to have large coefficients, making the computation of $K$ and $L$ more difficult, and the vast majority of curves in $S$ are likely to represent the generic case, where $\Gal(K/\Q)$ and $\Gal(L/\Q)$ are as large as possible.
In principle, one can control the isomorphism type of $\Gal(K/\Q)$ by placing appropriate constraints on the input parameters, but this is not enough to determine the Sato-Tate group, and it gives no control over $\Gal(L/\Q)$.

We instead adapted the search method used in \cite{FKRS}, generating $S$ by enumerating all curves of the form $y^2= \sum_{i=0}^6 c_ix^i$ satisfying coefficient bounds $|c_i|\le B_i$.
To quickly identify curves $C$ that are twists of $C^0$, we first compute $a_1(C)(p)$ for a handful of small primes $p$ that are inert in $M$, and immediately discard $C$ if $a_1(C)(p)\ne 0$ for any such $p$.
We then compute the absolute Igusa invariants of $C$, as defined in \cite{Igu}, and compare them to the corresponding values for $C^0$.
With the bounds $B_i$ chosen to encompass some $2^{50}$ curves with small coefficients, we obtain a set $S$ containing tens of thousand twists of $C^0$ in each case.

After applying the method in \S\ref{section: provisional T(C)} below to all of the curves in $S$, we had several candidate curves $C$ for every possible triple $T(C)$ that can arise when $C$ is defined over $\Q$ (the triples listed in Tables~\ref{Table: candidates2} and~\ref{Table: candidates3} that are not marked with an asterisk).  We then selected a single representative $C$ for each triple and computed $K$ and $L$ for each of these $C$, as described in \S\ref{section: computationKL}, and then computed the Galois groups $\Gal(L/\Q)$, $\Gal(K/\Q)$, and $\Gal(L/M)$, using the Magma computer algebra system \cite{Magma}, to obtain the true value of triple $T(C)$.
As expected, this computation confirmed the provisional value in every case.

\subsection{Provisional computation of \texorpdfstring{$T(C)$}{}}\label{section: provisional T(C)}

To provisionally identify the triple $T(C)$, we compute an approximation of the vector $z(C)$ (see Definition~\ref{definition: zvector}), which uniquely determines $T(C)$, by Theorem~\ref{theorem: main result}.
To do this, it suffices to determine the triples $(s,r,t)$ of residue degrees $(f_L(p),f_K(p),f_M(p))$ for a sample set of primes $p$ (say, primes $p\le 2^{16}$ of good reduction for $C$), and then count how often each triple appears.
The components $o(s,r)$ and $\bar{o}(s,r)$ of the vector $z(C)$ may be approximated by computing the relative frequencies of the triples $(s,r,1)$ and $(s,r,2)$, respectively, and normalizing so that $o(1,1)=1$.

We can easily compute $t=f_M(p)\in\{1,2\}$ by checking whether $p$ splits in~$M$, but we also need to compute $r=f_K(p)$ and $s=f_L(p)$, and we would like to do so \emph{without knowing $K$ or~$L$}.
This can be achieved as follows: we first compute $a_1(E^0)(p)$ and the values $a_1(C)(p)$ and $a_2(C)(p)$, as described in \S \ref{section: a1,a2 from a}, and then determine the unique map $F_{(s,r,t)}$ from Proposition~\ref{proposition: trthetaMEA} for which
\begin{equation}\label{eq:Fsrt}
F_{(s,r,t)}(a_1(E^0)(p)) = (\pm a_1(C)(p),a_2(C)(p)).
\end{equation}

\subsubsection{Computation of \texorpdfstring{$a_1(C)(p) \text{ and } a_2(C)(p)$}{}}\label{section: a1,a2 from a}
Efficient computation of $a_1(C)(p)$ and $a_2(C)(p)$  for an arbitrary genus~2 curve is addressed in \cite{KS08}, but in the special case of interest here, where $C$ is a $\Q$-twist of $C^0$, we use a faster approach.
The Jacobian of $C^0$ is $\Q$-isogenous to the square of $E^0$, an elliptic curve defined over $\Q$.  Because $E^0$ has complex multiplication, we can very efficiently determine $a_1(E)(p)$.  Taking $C^0_2$ as an example, $E^0_2$ is defined by the Weierstrass equation $y^2=x^3-5x^2-5x+1$.  This curve has CM by $M=\Q(\sqrt{-2})$, and for any prime $p>2$ we may compute $a=a_1(E^0_2)(p)$ as follows: $a=0$ if $p$ is inert in $M$ and otherwise $a=4x/\sqrt{p}$, where the integer $x$ satisfies $p=x^2+2y^2$ for some integer~$y$. The positive integer $z=|x|$ may be determined via Cornacchia's algorithm, and then $x=(-1)^\epsilon z$, where $\epsilon = (z-1)/2  + (p-1)(p+5)/16$; see \cite{RS} for details.
The computation for $C^0_3$ is similar: in this case $E^0_3$ is defined by $y^2=x^3+1$, with CM by $M=\Q(\sqrt{-3})$.

Having computed $a_1(E^0)(p)$, there are only a handful of pairs $(a_1,a_2)$ that are compatible with (\ref{eq:Fsrt}), that is, for which there exists a triple $(s,r,t)$ such that $F_{(s,r,t)}(a_1(E^0)(p)) = (\pm a_1, a_2)$.
Taking into account whether $C$ is a twist of $C^0_2$ or $C^0_3$, whether $p$ splits in $M$ or not, and that the sign of $a_1$ is actually ambiguous in only 2 cases, there are at most 8 possibilities.
Each compatible pair $(a_1,a_2)$ determines an integer
\[
n=p^2+p^{3/2}a_1+pa_2+p^{1/2}a_1+1,
\]
one of which is equal to $\#\Jac(C)(\F_p)$.
In most cases, if we pick a random point $P\in\Jac(C)(\F_p)$, the equation $nP=0$ will hold for exactly one $n$ and uniquely determine $a_1$ and $a_2$.
Even when this is not the case, after factoring the integers $n$ we can determine the order of any point $P$ in $\Jac(C)(\F_p)$, using just $\tilde{O}(\log p)$ operations in $\F_p$; see \cite[Ch.~7]{S07}.
This allows us to compute the order of $\Jac(C)(\F_p)$ using a probabilistic generic group algorithm (of Las Vegas type)  that runs in $O(p^{1/4})$ expected time; see \cite{S07} and \cite[Prop.~1]{KS08}.\footnote{The $O(p^{1/4})$ bound is a worst case estimate, it is faster than this for most $p$}.  This compares to an $O(p^{3/4})$ expected running time for an arbitrary genus 2 curve using a generic group algorithm.\footnote{As noted in \cite{KS08}, the asymptotically faster polynomial-time algorithm of Pila \cite{Pila90} is not practically useful in the range of $p$ relevant to the computations considered here.}

Having computed $L_p(C,1)=\#\Jac(C)(\F_p)$, we use the same method to determine $L_p(C,-1)=\#\Jac(\tilde{C})(\F-p)$, where $\tilde{C}$ is any non-trivial quadratic twist of $C$ over~$\F_p$, and these two values uniquely determines $a_1$ and $a_2$.

The algorithm described above is included in the most recent version of the \texttt{smalljac} software library, whose source code is available at \cite{S11b}.

\subsection{Computation of  \texorpdfstring{$K$}{} and  \texorpdfstring{$L$}{}} \label{section: computationKL}

In this section, we describe the procedure used to compute the fields $K$ and $L$ for the curves $C$ listed in Tables \ref{Table: curves2} and \ref{Table: curves3}.

For the field $K$, its characterization in Lemma \ref{lemma: automorphisms vs endomorphisms} as the minimal field over which all the automorphisms of $C$ are defined turns out to be the most computationally effective. For all 41 curves $C\colon y^2=f(x)$ listed in Tables~\ref{Table: curves2} and \ref{Table: curves3}, one readily checks that $\Aut(C^0_\Qbar)\simeq\Aut(C_\Qbar)=\Aut(C_{F(\zeta_{24})})$, where $F$ is the splitting field of $f(x)$ (see Remark \ref{remark: in general} below). It is then a finite problem to identify the minimal subfield $K$ of $F(\zeta_{24})$ for which $\Aut(C_K)=\Aut(C_{F(\zeta_{24})})$.

Having computed $K$, we determine $L$ as follows.
For any non-hyperelliptic involution $\beta\in\Aut(C^0_M)$, the elliptic quotient $C^0/\langle\beta\rangle$ is defined over~$M$.
If $\beta_1$ and $\beta_2$ are conjugate in $\Aut(C^0_M)$, then $C^0/{\langle \beta_1\rangle }\simeq C^0/\langle\beta_2\rangle$.
For $C^0_2$ there is just one conjugacy class of non-hyperelliptic involutions, hence in this case every elliptic quotient $C^0/\langle\beta\rangle$ is isomorphic to $E^0_M$.
For $C^0_3$ there are two conjugacy classes of non-hyperelliptic involutions, of size 2 and 6 (see Table~\ref{table: charactertable3}).
The first corresponds to the $M$-isomorphism class of $E^0_3$, and the second corresponds to the $M$-isomorphism class of the elliptic curve $y^2=x^3-15x+22$.

Since we know $K$ explicitly, we can compute $\Aut(C_K)$ and enumerate all the non-hyperelliptic involutions $\alpha$ (there are 12 when $d=2$ and 8 when $d=3$).
For $d=2$, we pick any $\alpha$, and for $d=3$ we pick $\alpha$ from the conjugacy class of size 2.
Define $\tilde{E}:=C_K/\langle \alpha\rangle$ and $\tilde E^0 := C^0_M/\langle\phi\alpha\phi^{-1}\rangle$.
The isomorphism $\phi$ induces an isomorphism $\tilde\phi\colon\tilde E_L\rightarrow \tilde E^0_L$.
As in the proof of Proposition \ref{proposition: degree}, $L$ is the compositum of $K$ and the minimal field over which $\tilde\phi$ is defined.
Our choice of $\alpha$ insures that $\tilde E^0\simeq E^0_M$, thus $\tilde{E}_L\simeq E^0_L$.

By applying \cite[Lemma 2.2]{CGLR99}, we can compute an explicit Weierstrass equation for $\tilde E$ of the form
$$\tilde E\colon Y^2=X^3+AX+B,\qquad \text{with $A$, $B\in K$,}$$
Writing $E^0$ in the form $Y^2=X^3+UX+V$, there then exists $\gamma\in L$ such that $U=\gamma^4 A$ and $ V=\gamma^6 B$, and $\gamma$ generates $L$ as an (at most quadratic) extension of $K$.  We can easily derive $\gamma$ from the coefficients $A$, $B$, $U$, and $V$.

\begin{remark}\label{remark: in general} In fact, it is true in general that for any twist $C$ of $C^0_2$ (resp. $C^0_3$), the field $K$ is contained in $F(\sqrt{-2})$ (resp. $F(\sqrt{-3},i)$). We thank J. Quer for kindly providing the following argument. 

Let $\Aut(C_\Qbar)^*$ denote the subgroup of $\Aut(C_\Qbar)$ generated by those elements~$\alpha$ such that $\Trace(\iota(\alpha))$ is nonzero. We claim that $\Aut(C_\Qbar)^*=\Aut(C_{FM})^*$. Let $\WP(C)$ denote the set of Weierstrass points of $C$ and let $\sigma$ be an element of $G_{FM}$. It suffices to show that ${}^ \sigma\alpha=\alpha$ for every $\alpha$ in   $\Aut(C_\Qbar)^*$ such that $\Trace(\iota(\alpha))$ is nonzero. Observe that for every $P$ in $\WP(C)$, one has ${}^{\sigma}P=P$. Then, writing $Q=\alpha^ {-1}(P)$, we have
$$
^{\sigma}\alpha\circ\alpha^{-1}(P)=({}^{\sigma}\alpha)(Q)={}^{\sigma}(\alpha(Q))={}^{\sigma}P=P\,,
$$
which implies that ${}^{\sigma}\alpha$ is either $\alpha$ or $w\alpha$,
since the action of $\Aut(C_\Qbar)/\langle w\rangle$ on $\WP(C)$ is faithful. Provided that $\Trace(\iota(\alpha))$ is in $M$, the latter option is not possible, since otherwise we would have
$$
\Trace(\iota(\alpha))={}^{\sigma}\Trace(\iota(\alpha))=\Trace(\iota(w\alpha))=-\Trace(\iota(\alpha))\,,
$$
contradicting the fact that $\Trace(\iota(\alpha))$ is nonzero. Since $\Aut(C_\Qbar)$ and $\Aut(C^0_\Qbar)$ are conjugated, the groups $\Aut(C_\Qbar)^*$ and $\Aut(C^0_\Qbar)^*$ are isomorphic. It is straightforward to check that 
$$
\Aut((C^0_2)_\Qbar)^*\simeq \symtilde 4\qquad 
\text{and}\qquad
\Aut((C^0_3)_\Qbar)^*\simeq \cyc 2\times \cyc 6\,.
$$ 
Thus, for every twist $C$ of $C^0_2$, the field $K$ is contained in $F(\sqrt{-2})$; but for a twist $C$ of $C^0_3$ the order of $\Aut((C^0_3)_{F(\sqrt{-3})})$ can be $12$ or $24$. By considering the  parameterizations given by Cardona \cite[Prop. 7.4.1]{Car01} of all the twists $C$ of $C^0_3$ as well as of the corresponding embeddings $\iota(\Aut(C_\Qbar))$ in $\GL_2(\Qbar)$, one may explicitly verify that $K$ is always contained in $F(\sqrt{-3},i)$. 
\end{remark}

\subsubsection{An example}
Consider the twist $C$ of $C^0_3$ defined by the hyperelliptic equation
$$y^2 = f(x)= x^6 + 15x^4 + 20x^3 + 30x^2 + 18x + 5\,$$
over $\Q$.
This curve is listed in Table~\ref{Table: curves3} for the triple $[\langle 24,5\rangle,\langle 12,4\rangle,\langle 12,1\rangle]$.
Let us prove that this is in fact the triple $T(C)=[\Gal(L/\Q),\Gal(K/\Q),\Gal(L/M)]$.

We first compute $K$.
Let $F$ denote the splitting field of $f(x)$.
One checks (via Magma) that $|\Aut(C_{MF})|=24$, where $M=\Q(\sqrt{-3})$, and therefore $K\subseteq MF$ (since we know \emph{a priori} that $|\Aut(C_K)|=|\Aut((C^0_3)_\Qbar)|=24$).
By enumerating the various subfields of $MF$, we find that the minimal subfield $K$ of $MF$ for which $|\Aut(C_K)|=24$ is $K=M(\sqrt 5,a)$, where $a^3+3a-1=0$.

To compute $L$, we choose the non-hyperelliptic involution $\alpha$ of $\Aut(C_K)$ whose image under the map $\iota\colon \Aut(C_{\Qbar})\rightarrow \GL_2(\Qbar)$ defined in (\ref{equation: injection}) is
$$\iota(\alpha)=
\frac{1}{5}\begin{pmatrix}
\sqrt{5} & -2\sqrt{5}\\
-2\sqrt{5} & -\sqrt{5}
\end{pmatrix}\,.
$$
Applying \cite[Lemma 2.2]{CGLR99} yields a Weierstrass equation for $\tilde E=C/\langle \alpha\rangle$:
$$\tilde E\colon Y^2=X^3 +B, \qquad\text{with $B=-\frac{11}{97656250}\sqrt{5} + \frac{1}{3906250}$}\,.$$
Since $E^0$ is the curve $y^2=x^3+1$, we have $U=0$ and $V=1$, so $\gamma^6=1/B$.
This implies that
$$\gamma^2 - \left(\frac{125}{2}\sqrt{5} + \frac{375}{2}\right)a^2 +\left(\frac{125}{2}\sqrt{5} + \frac{125}{2}\right)a - 125\sqrt{5} - 375=0,$$
and one finds that $L=K(\sqrt{2\sqrt{5}+10})$.

Having explicitly computed the fields $K$ and $L$, it is then straight-forward to verify that $\Gal(L/\Q)\simeq\langle 24,5\rangle$, $\Gal(K/\Q)\simeq\langle 12,4\rangle$, and $\Gal(L/M)\simeq\langle 12,1\rangle$ using Magma.

\section{Tables}\label{section: tables}
This section contains tables described in earlier sections, whose definitions we briefly recall. Remember that $C^0$ is one of the two curves $C^0_3\colon y^2=x^6+1$ (in which case $M=\Q(\sqrt{-3}$) or $C^0_2\colon y^2=x^6-5x^4-5x^2+1$ (in which case $M=\Q(\sqrt{-2}$).
Tables \ref{table: charactertable2} and \ref{table: charactertable3} are the character tables of $\Aut((C^0_3)_\Qbar)$ and $\Aut((C^0_2)_\Qbar)$.
Tables \ref{Table: candidates2} and \ref{Table: candidates3} list (up to isomorphism) the possible values of the triples $T(C)$ that can arise when $C$ is a $\Q$-twist of the curve $C^0$.

The computation of these tables is described in \S\ref{section: triples of Galois groups}.
Each triple $[H,H/N,H_0]$ is a possible value for $T(C)=[\Gal(L/\Q),\Gal(K/\Q,\Gal(L/M)]$, and is determined by a subgroup $H\subset G_{C^0}$ whose intersection with $\Aut(C_M^0)$ is an index 2 subgroup $H_0$ of $H$, where $N=H\cap Z(G_{C^0})$.

For each triple $T(C)$ we list the corresponding Sato-Tate group $G$ and its matching Galois type, as defined in \cite{FKRS}, as well as the vector $z(C)$ given by Definition \ref{definition: zvector}, all of which are uniquely determined by $T(C)$, by Theorem \ref{theorem: main result}.  As proven in \cite{FKRS}, the Sato-Tate groups $J(C_1)$, $J(C_3)$, and $C_{4,1}$ cannot arise for a genus 2 curve defined over $\Q$, and the corresponding rows in Tables \ref{Table: candidates2} and \ref{Table: candidates3} are marked with an asterisk.

In Tables \ref{Table: curves3} and \ref{Table: curves2} we list representative curves that realize every triple $T(C)$ that can occur when $C$ is defined over $\Q$.  For each curve we also give an explicit description of the fields $K$ and $L$, where $K$ is the minimal field for which $\Aut(C_K)=\Aut(C_\Qbar)$, and $L$ is the minimal extension of $K$ over which $C$ is isomorphic to $C^0$.
The methods used to obtain these curves and the computation of $K$ and $L$ are described in \S\ref{section: computations}.

\begin{table}[h]
\begin{center}
\caption{Character Table of $\Aut((C^0_2)_M)\simeq \langle 48,29\rangle$}
\label{table: charactertable2}
\vspace{6pt}
\begin{tabular}{rrrrrrrrr}
Class &  1a & 2a & 2b & 3a & 4a & 6a &  8a &  8b \\\hline
Size  &   1 & 1 & 12 & 8 & 6 & 8 &  6  & 6\\\hline
$\chi_1$  &  1 & 1 & 1 & 1 & 1 & 1  & 1 &  1\\
$\chi_2$  &  1 & 1 &-1 & 1 & 1 & 1 & -1 & -1\\
$\chi_3$  & 2 & 2 & 0 & -1 & 2 &-1 &  0 &  0\\
$\chi_4$  &  2 &-2 & 0 &-1 & 0 & 1 & $\sqrt{-2}$ & $-\sqrt{-2}$\\
$\chi_5$  &  2 &-2 & 0 &-1 & 0 & 1 & $-\sqrt{-2}$ & $\sqrt{-2}$\\
$\chi_6$  & 3 & 3  &1  &0  &-1 & 0 & -1 & -1\\
$\chi_7$  & 3 & 3  & -1 & 0 &-1&  0 &  1 &  1\\
$\chi_8$  & 4 &-4  &0   & 1 & 0& -1 &  0 &  0\\\hline
\end{tabular}
\bigskip
\caption{Character Table of $\Aut((C^0_3)_M)\simeq \langle 24,8\rangle$}
\label{table: charactertable3}
\vspace{6pt}
\begin{tabular}{rrrrrrrrrr}
Class &  1a & 2a & 2b & 2c & 3a & 4a  &   6a   &  6b & 6c\\\hline
Size  &   1 & 1 & 2 & 6 & 2 & 6  &   2  &   2 & 2\\\hline
$\chi_1$  & 1 & 1 & 1 & 1 & 1 & 1  &  1 & 1 & 1\\
$\chi_2$  & 1 & 1 & 1 &-1 & 1 &-1  &  1 & 1 & 1\\
$\chi_3$  &  1 & 1 & -1 & -1 &  1 &  1 & -1 &  -1 &  1\\
$\chi_4$  &  1 & 1 & -1 & 1 & 1 & -1 & -1 & -1 & 1\\
$\chi_5$  &  2 & 2 & -2 &  0 & -1 &  0  &  1 & 1 & -1\\
$\chi_6$  & 2  &-2&  0&  0 & 2 & 0  &   0  &   0 &-2\\
$\chi_7$ & 2 &  2 &  2 &  0 & -1 &  0 &   -1 &   -1& -1\\
$\chi_8$ &  2 &-2 & 0 & 0 & -1 & 0 & $-\sqrt{-3}$ & $\sqrt{-3}$ & 1\\
$\chi_9$ & 2 & -2 & 0 & 0 & -1 & 0 & $\sqrt{-3}$ & $-\sqrt{-3}$ & 1\\\hline
\end{tabular}

\end{center}
\end{table}

\begin{table}
\begin{center}
\caption{Triples for twists of $C^0_2$.}\label{Table: candidates2}
\footnotesize
\vspace{-4pt}
\begin{tabular}{llllll}
$G$ & $H$ & $H/N$ & $H_0$ & Galois type & $z(H,N,H_0)$\\\hline
$*J(C_1)$ & $\langle 4,1 \rangle$ & $\langle 2,1 \rangle$ & $\langle 2,1 \rangle$ & $\bF[\cyc{2},\cyc{1},\HH]$ & $[ 1, 1, 0, 0, 0, 0, 0, 0, 0, 0, 2, 0, 0, 0 ]$\\
$J(C_2)$ & $\langle 8,2 \rangle$ & $\langle 4,2 \rangle$ & $\langle 4,1 \rangle$ & $\bF[\dih{2},\cyc{2},\HH]$ & $[ 1, 1, 0, 0, 2, 0, 0, 0, 0, 2, 2, 0, 0, 0 ]$\\
$J(C_2)$ & $\langle 8,3 \rangle$ & $\langle 4,2 \rangle$ & $\langle 4,2 \rangle$ & $\bF[\dih{2},\cyc{2},\HH]$ & $[ 1, 1, 2, 0, 0, 0, 0, 0, 0, 2, 2, 0, 0, 0 ]$\\
$*J(C_3)$ & $\langle 12,2 \rangle$ & $\langle 6,2 \rangle$ & $\langle 6,2 \rangle$ & $\bF[\cyc{6},\cyc{3},\HH]$ & $[ 1, 1, 0, 2, 0, 2, 0, 0, 0, 0, 2, 0, 0, 4 ]$\\
$J(C_4)$ & $\langle 16,6 \rangle$ & $\langle 8,2 \rangle$ & $\langle 8,1 \rangle$ & $\bF[\cyc{4}\times\cyc{2},\cyc{4}]$ & $[ 1, 1, 0, 0, 2, 0, 0, 4, 0, 2, 2, 0, 4, 0 ]$\\
$J(D_2)$ & $\langle 16,11 \rangle$ & $\langle 8,5 \rangle$ & $\langle 8,3 \rangle$ & $\bF[\dih{2}\times\cyc{2},\dih{2}]$ & $[ 1, 1, 4, 0, 2, 0, 0, 0, 0, 6, 2, 0, 0, 0 ]$\\
$J(D_2)$ & $\langle 16,13 \rangle$ & $\langle 8,5 \rangle$ & $\langle 8,4 \rangle$ & $\bF[\dih{2}\times\cyc{2},\dih{2}]$ & $[ 1, 1, 0, 0, 6, 0, 0, 0, 0, 6, 2, 0, 0, 0 ]$\\
$J(D_3)$ & $\langle 24,6 \rangle$ & $\langle 12,4 \rangle$ & $\langle 12,4 \rangle$ & $\bF[\dih{6},\dih{3},\HH]$ & $[ 1, 1, 6, 2, 0, 2, 0, 0, 0, 6, 2, 0, 0, 4 ]$\\
$J(D_4)$ & $\langle 32,43 \rangle$ & $\langle 16,11 \rangle$ & $\langle 16,8 \rangle$ & $\bF[\dih{4}\times\cyc{2},\dih{4}]$ & $[ 1, 1, 4, 0, 6, 0, 0, 4, 0, 10, 2, 0, 4, 0 ]$\\
$J(T)$ & $\langle 48,33 \rangle$ & $\langle 24,13 \rangle$ & $\langle 24,3 \rangle$ & $\bF[\alt{4}\times\cyc{2},\alt{4}]$ & $[ 1, 1, 0, 8, 6, 8, 0, 0, 0, 6, 2, 0, 0, 16 ]$\\
$J(O)$ & $\langle 96,193 \rangle$ & $\langle 48,48 \rangle$ & $\langle 48,29 \rangle$ & $\bF[\sym{4}\times\cyc{2},\sym{4}]$ & $[ 1, 1, 12, 8, 6, 8, 0, 12, 0, 18, 2, 0, 12, 16 ]$\\
$C_{2,1}$ & $\langle 2,1 \rangle$ & $\langle 2,1 \rangle$ & $\langle 1,1 \rangle$ & $\bF[\cyc{2},\cyc{1},\M_2(\R)]$ & $[ 1, 0, 0, 0, 0, 0, 0, 0, 0, 1, 0, 0, 0, 0 ]$\\
$C_{2,1}$ & $\langle 4,2 \rangle$ & $\langle 2,1 \rangle$ & $\langle 2,1 \rangle$ & $\bF[\cyc{2},\cyc{1},\M_2(\R)]$ & $[ 1, 1, 0, 0, 0, 0, 0, 0, 0, 2, 0, 0, 0, 0 ]$\\
$*C_{4,1}$ & $\langle 8,1 \rangle$ & $\langle 4,1 \rangle$ & $\langle 4,1 \rangle$ & $\bF[\cyc{4},\cyc{2}]$ & $[ 1, 1, 0, 0, 2, 0, 0, 0, 0, 0, 0, 0, 4, 0 ]$\\
$D_{2,1}$ & $\langle 4,2 \rangle$ & $\langle 4,2 \rangle$ & $\langle 2,1 \rangle$ & $\bF[\dih{2},\cyc{2},\M_2(\R)]$ & $[ 1, 0, 1, 0, 0, 0, 0, 0, 0, 2, 0, 0, 0, 0 ]$\\
$D_{2,1}$ & $\langle 8,3 \rangle$ & $\langle 4,2 \rangle$ & $\langle 4,1 \rangle$ & $\bF[\dih{2},\cyc{2},\M_2(\R)]$ & $[ 1, 1, 0, 0, 2, 0, 0, 0, 0, 4, 0, 0, 0, 0 ]$\\
$D_{2,1}$ & $\langle 8,5 \rangle$ & $\langle 4,2 \rangle$ & $\langle 4,2 \rangle$ & $\bF[\dih{2},\cyc{2},\M_2(\R)]$ & $[ 1, 1, 2, 0, 0, 0, 0, 0, 0, 4, 0, 0, 0, 0 ]$\\
$D_{3,2}$ & $\langle 6,1 \rangle$ & $\langle 6,1 \rangle$ & $\langle 3,1 \rangle$ & $\bF[\dih{3},\cyc{3}]$ & $[ 1, 0, 0, 2, 0, 0, 0, 0, 0, 3, 0, 0, 0, 0 ]$\\
$D_{3,2}$ & $\langle 12,4 \rangle$ & $\langle 6,1 \rangle$ & $\langle 6,2 \rangle$ & $\bF[\dih{3},\cyc{3}]$ & $[ 1, 1, 0, 2, 0, 2, 0, 0, 0, 6, 0, 0, 0, 0 ]$\\
$D_{4,1}$ & $\langle 16,7 \rangle$ & $\langle 8,3 \rangle$ & $\langle 8,3 \rangle$ & $\bF[\dih{4},\dih{2}]$ & $[ 1, 1, 4, 0, 2, 0, 0, 0, 0, 4, 0, 0, 4, 0 ]$\\
$D_{4,1}$ & $\langle 16,8 \rangle$ & $\langle 8,3 \rangle$ & $\langle 8,4 \rangle$ & $\bF[\dih{4},\dih{2}]$ & $[ 1, 1, 0, 0, 6, 0, 0, 0, 0, 4, 0, 0, 4, 0 ]$\\
$D_{4,2}$ & $\langle 16,7 \rangle$ & $\langle 8,3 \rangle$ & $\langle 8,1 \rangle$ & $\bF[\dih{4},\cyc{4}]$ & $[ 1, 1, 0, 0, 2, 0, 0, 4, 0, 8, 0, 0, 0, 0 ]$\\
$O_1$ & $\langle 48,29 \rangle$ & $\langle 24,12 \rangle$ & $\langle 24,3 \rangle$ & $\bF[\sym{4},\alt{4}]$ & $[ 1, 1, 0, 8, 6, 8, 0, 0, 0, 12, 0, 0, 12, 0 ]$\\\hline
\end{tabular}
\end{center}
\end{table}

\begin{table}
\setlength{\extrarowheight}{1pt}
\begin{center}
\vspace{-10pt}
\caption{Triples for twists of $C^0_3$.}\label{Table: candidates3}
\footnotesize
\vspace{-4pt}
\begin{tabular}{llllll}
$G$ & $H$ & $H/N$ & $H_0$ & Galois type & $z(H,N,H_0)$\\\hline
$*J(C_1)$ & $\langle 4,1 \rangle$ & $\langle 2,1 \rangle$ & $\langle 2,1 \rangle$ & $\bF[\cyc{2},\cyc{1},\HH]$ & $[ 1, 1, 0, 0, 0, 0, 0, 0, 0, 0, 2, 0, 0, 0 ]$\\
$J(C_2)$ & $\langle 8,2 \rangle$ & $\langle 4,2 \rangle$ & $\langle 4,1 \rangle$ & $\bF[\dih{2},\cyc{2},\HH]$ & $[ 1, 1, 0, 0, 2, 0, 0, 0, 0, 2, 2, 0, 0, 0 ]$\\
$J(C_2)$ & $\langle 8,3 \rangle$ & $\langle 4,2 \rangle$ & $\langle 4,2 \rangle$ & $\bF[\dih{2},\cyc{2},\HH]$ & $[ 1, 1, 2, 0, 0, 0, 0, 0, 0, 2, 2, 0, 0, 0 ]$\\
$*J(C_3)$ & $\langle 12,2 \rangle$ & $\langle 6,2 \rangle$ & $\langle 6,2 \rangle$ & $\bF[\cyc{6},\cyc{3},\HH]$ & $[ 1, 1, 0, 2, 0, 2, 0, 0, 0, 0, 2, 0, 0, 4 ]$\\
$J(C_6)$ & $\langle 24,10 \rangle$ & $\langle 12,5 \rangle$ & $\langle 12,5 \rangle$ & $\bF[\cyc{6}\times\cyc{2},\cyc{6}]$ & $[ 1, 1, 2, 2, 0, 2, 4, 0, 0, 2, 2, 4, 0, 4 ]$\\
$J(D_2)$ & $\langle 16,11 \rangle$ & $\langle 8,5 \rangle$ & $\langle 8,3 \rangle$ & $\bF[\dih{2}\times\cyc{2},\dih{2}]$ & $[ 1, 1, 4, 0, 2, 0, 0, 0, 0, 6, 2, 0, 0, 0 ]$\\
$J(D_3)$ & $\langle 24,5 \rangle$ & $\langle 12,4 \rangle$ & $\langle 12,1 \rangle$ & $\bF[\dih{6},\dih{3},\HH]$ & $[ 1, 1, 0, 2, 6, 2, 0, 0, 0, 6, 2, 0, 0, 4 ]$\\
$J(D_3)$ & $\langle 24,6 \rangle$ & $\langle 12,4 \rangle$ & $\langle 12,4 \rangle$ & $\bF[\dih{6},\dih{3},\HH]$ & $[ 1, 1, 6, 2, 0, 2, 0, 0, 0, 6, 2, 0, 0, 4 ]$\\
$J(D_6)$ & $\langle 48,38 \rangle$ & $\langle 24,14 \rangle$ & $\langle 24,8 \rangle$ & $\bF[\dih{6}\times\cyc{2},\dih{6}]$ & $[ 1, 1, 8, 2, 6, 2, 4, 0, 0, 14, 2, 4, 0, 4 ]$\\
$C_{2,1}$ & $\langle 2,1 \rangle$ & $\langle 2,1 \rangle$ & $\langle 1,1 \rangle$ & $\bF[\cyc{2},\cyc{1},\M_2(\R)]$ & $[ 1, 0, 0, 0, 0, 0, 0, 0, 0, 1, 0, 0, 0, 0 ]$\\
$C_{2,1}$ & $\langle 4,2 \rangle$ & $\langle 2,1 \rangle$ & $\langle 2,1 \rangle$ & $\bF[\cyc{2},\cyc{1},\M_2(\R)]$ & $[ 1, 1, 0, 0, 0, 0, 0, 0, 0, 2, 0, 0, 0, 0 ]$\\
$C_{6,1}$ & $\langle 6,2 \rangle$ & $\langle 6,2 \rangle$ & $\langle 3,1 \rangle$ & $\bF[\cyc{6},\cyc{3},\M_2(\R)]$ & $[ 1, 0, 0, 2, 0, 0, 0, 0, 0, 1, 0, 2, 0, 0 ]$\\
$C_{6,1}$ & $\langle 12,5 \rangle$ & $\langle 6,2 \rangle$ & $\langle 6,2 \rangle$ & $\bF[\cyc{6},\cyc{3},\M_2(\R)]$ & $[ 1, 1, 0, 2, 0, 2, 0, 0, 0, 2, 0, 4, 0, 0 ]$\\
$D_{2,1}$ & $\langle 4,2 \rangle$ & $\langle 4,2 \rangle$ & $\langle 2,1 \rangle$ & $\bF[\dih{2},\cyc{2},\M_2(\R)]$ & $[ 1, 0, 1, 0, 0, 0, 0, 0, 0, 2, 0, 0, 0, 0 ]$\\
$D_{2,1}$ & $\langle 8,3 \rangle$ & $\langle 4,2 \rangle$ & $\langle 4,1 \rangle$ & $\bF[\dih{2},\cyc{2},\M_2(\R)]$ & $[ 1, 1, 0, 0, 2, 0, 0, 0, 0, 4, 0, 0, 0, 0 ]$\\
$D_{2,1}$ & $\langle 8,5 \rangle$ & $\langle 4,2 \rangle$ & $\langle 4,2 \rangle$ & $\bF[\dih{2},\cyc{2},\M_2(\R)]$ & $[ 1, 1, 2, 0, 0, 0, 0, 0, 0, 4, 0, 0, 0, 0 ]$\\
$D_{3,2}$ & $\langle 6,1 \rangle$ & $\langle 6,1 \rangle$ & $\langle 3,1 \rangle$ & $\bF[\dih{3},\cyc{3}]$ & $[ 1, 0, 0, 2, 0, 0, 0, 0, 0, 3, 0, 0, 0, 0 ]$\\
$D_{3,2}$ & $\langle 12,4 \rangle$ & $\langle 6,1 \rangle$ & $\langle 6,2 \rangle$ & $\bF[\dih{3},\cyc{3}]$ & $[ 1, 1, 0, 2, 0, 2, 0, 0, 0, 6, 0, 0, 0, 0 ]$\\
$D_{6,1}$ & $\langle 12,4 \rangle$ & $\langle 12,4 \rangle$ & $\langle 6,1 \rangle$ & $\bF[\dih{6},\dih{3},\M_2(\R)]$ & $[ 1, 0, 3, 2, 0, 0, 0, 0, 0, 4, 0, 2, 0, 0 ]$\\
$D_{6,1}$ & $\langle 24,8 \rangle$ & $\langle 12,4 \rangle$ & $\langle 12,1 \rangle$ & $\bF[\dih{6},\dih{3},\M_2(\R)]$ & $[ 1, 1, 0, 2, 6, 2, 0, 0, 0, 8, 0, 4, 0, 0 ]$\\
$D_{6,1}$ & $\langle 24,14 \rangle$ & $\langle 12,4 \rangle$ & $\langle 12,4 \rangle$ & $\bF[\dih{6},\dih{3},\M_2(\R)]$ & $[ 1, 1, 6, 2, 0, 2, 0, 0, 0, 8, 0, 4, 0, 0 ]$\\
$D_{6,2}$ & $\langle 12,4 \rangle$ & $\langle 12,4 \rangle$ & $\langle 6,2 \rangle$ & $\bF[\dih{6},\cyc{6}]$ & $[ 1, 0, 1, 2, 0, 0, 2, 0, 0, 6, 0, 0, 0, 0 ]$\\
$D_{6,2}$ & $\langle 24,14 \rangle$ & $\langle 12,4 \rangle$ & $\langle 12,5 \rangle$ & $\bF[\dih{6},\cyc{6}]$ & $[ 1, 1, 2, 2, 0, 2, 4, 0, 0, 12, 0, 0, 0, 0 ]$\\\hline
\end{tabular}
\end{center}
\end{table}
\bigskip
\medskip

\begin{table}
\begin{center}
\caption{Twists of $C^0_2\colon y^2=x^6-5x^4-5x^2+1$ realizing each triple.}\label{Table: curves2}
\vspace{8pt}
\scriptsize
\setlength{\extrarowheight}{1pt}
\begin{tabular}{llll}
$G$ & $[\Gal(L/\Q), \Gal(K/\Q), \Gal(L/M)]$ & $K$ & $L$\\\hline
\multicolumn{4}{l}{$\boldsymbol{y^2 = x^5-x}$}\\
$J(C_2)$ & $\langle 8,2 \rangle$, $\langle 4,2 \rangle$, $\langle 4,1 \rangle$ & $M(i)$ & $K\Bigl(\sqrt{\sqrt{2}+2}\Bigr)$\vspace{1pt}\\
\multicolumn{4}{l}{$\boldsymbol{y^2 = x^5+4x}$}\\
$J(C_2)$ & $\langle 8,3 \rangle$, $\langle 4,2 \rangle$, $\langle 4,2 \rangle$ &$M(i)$ & $K(\sqrt[4]{2})$\vspace{4pt}\\
\multicolumn{4}{l}{$\boldsymbol{y^2 = x^6 + x^5 - 5x^4 - 5x^2 - x + 1}$}\\
$J(C_4)$ & $\langle 16,6 \rangle$, $\langle 8,2 \rangle$, $\langle 8,1 \rangle$ & $M(\sqrt{\sqrt{17}+17})$ & $K\Bigl(\sqrt{(\sqrt{17}+3)\sqrt{\sqrt{17}+17}-8\sqrt{17}}\Bigr)$\vspace{1pt}\\
\multicolumn{4}{l}{$\boldsymbol{y^2 = x^5 + 9x}$}\\
$J(D_2)$ & $\langle 16,11 \rangle$, $\langle 8,5 \rangle$, $\langle 8,3 \rangle$ & $M(i,\sqrt{3})$ & $K(\sqrt[4]{3})$\vspace{4pt}\\
\multicolumn{4}{l}{$\boldsymbol{y^2 =x^5 - 9x}$}\\
$J(D_2)$ & $\langle 16,13 \rangle$, $\langle 8,5 \rangle$, $\langle 8,4 \rangle$ & $M(i,\sqrt{3})$ & $K(\sqrt[4]{3i})$\vspace{4pt}\\
\multicolumn{4}{l}{$\boldsymbol{y^2 =x^6 + 10x^3 - 2}$}\\
$J(D_3)$ & $\langle 24,6 \rangle$, $\langle 12,4 \rangle$, $\langle 12,4 \rangle$ & $M(\sqrt{-3},\sqrt[3]{-2})$ & $K\Bigl(\sqrt{\sqrt{6}-2}\Bigr)$\vspace{1pt}\\
\multicolumn{4}{l}{$\boldsymbol{y^2 =x^5 + 3x}$}\\
$J(D_4)$ & $\langle 32,43 \rangle$, $\langle 16,11 \rangle$, $\langle 16,8 \rangle$ & $M(i,\sqrt[4]{3})$ & $K(\sqrt[8]{3})$\vspace{4pt}\\
\multicolumn{4}{l}{$\boldsymbol{y^2 =x^6 + 6x^5 - 20x^4 + 20x^3 - 20x^2 - 8x + 8}$}\\
$J(T)$ & $\langle 48,33 \rangle$, $\langle 24,13 \rangle$, $\langle 24,3 \rangle$ & $M(u_1,u_2)$ & $K(\sqrt{v_1})$\vspace{4pt}\\
\multicolumn{4}{l}{$\boldsymbol{y^2 = x^6 - 5x^4 + 10x^3 - 5x^2 + 2x - 1}$}\\
$J(O)$ & $\langle 96,193 \rangle$, $\langle 48,48 \rangle$, $\langle 48,29 \rangle$ & $M(\sqrt{-11},u_3,u_4),$ & $K(\sqrt{v_2})$\vspace{4pt}\\
\multicolumn{4}{l}{$\boldsymbol{y^2 = x^6 - 5x^4 - 5x^2 + 1}$}\\
$C_{2,1}$ & $\langle 2,1 \rangle$, $\langle 2,1 \rangle$, $\langle 1,1 \rangle$ & $M$ & $K$\vspace{4pt}\\
\multicolumn{4}{l}{$\boldsymbol{y^2 =-x^6+5x^4+5x^2-1}$}\\
$C_{2,1}$ & $\langle 4,2 \rangle$, $\langle 2,1 \rangle$, $\langle 2,1 \rangle$ & $M$ & $K(i)$\vspace{4pt}\\
\multicolumn{4}{l}{$\boldsymbol{y^2 =x^5 + x}$}\\
$D_{2,1}$ & $\langle 4,2 \rangle$, $\langle 4,2 \rangle$, $\langle 2,1 \rangle$ & $M(i)$ & $K$\vspace{4pt}\\
\multicolumn{4}{l}{$\boldsymbol{y^2 =x^6 + 3x^5 - 20x^4 + 30x^3 - 35x^2 + 3x + 10}$}\\
$D_{2,1}$ & $\langle 8,3 \rangle$, $\langle 4,2 \rangle$, $\langle 4,1 \rangle$ & $M(\sqrt{7})$ & $K\Bigr(\sqrt{3\sqrt{7}+7}\Bigl)$\vspace{1pt}\\
\multicolumn{4}{l}{$\boldsymbol{y^2 =x^5 + 81x}$}\\
$D_{2,1}$ & $\langle 8,5 \rangle$, $\langle 4,2 \rangle$, $\langle 4,2 \rangle$ & $M(i)$ & $K(\sqrt{3})$\vspace{4pt}\\
\multicolumn{4}{l}{$\boldsymbol{y^2 =x^6 - 18x^5 - 15x^4 - 20x^3 + 135x^2 - 498x - 89}$}\\
$D_{3,2}$ & $\langle 6,1 \rangle$, $\langle 6,1 \rangle$, $\langle 3,1 \rangle$ & $M(u_5)$ & $K$\vspace{4pt}\\
\multicolumn{4}{l}{$\boldsymbol{y^2 =x^6 + 4x^5 - 10x^4 + 80x^3 + 140x^2 + 144x - 184}$}\\
$D_{3,2}$ & $\langle 12,4 \rangle$, $\langle 6,1 \rangle$, $\langle 6,2 \rangle$ & $M(u_6)$ & $K(i)$\vspace{4pt}\\
\multicolumn{4}{l}{$\boldsymbol{y^2 =x^5 - 2x}$}\\
$D_{4,1}$ & $\langle 16,7 \rangle$, $\langle 8,3 \rangle$, $\langle 8,3 \rangle$ & $M(i\sqrt[4]{-2})$ & $K(\sqrt[8]{-2})$\vspace{4pt}\\
\multicolumn{4}{l}{$\boldsymbol{y^2 =x^5 + 2x}$}\\
$D_{4,1}$ & $\langle 16,8 \rangle$, $\langle 8,3 \rangle$, $\langle 8,4 \rangle$ & $M(i\sqrt[4]{2})$ & $K(\sqrt[8]{2})$\vspace{4pt}\\
\multicolumn{4}{l}{$\boldsymbol{y^2 =x^6 + x^5 + 10x^3 + 5x^2 + x - 2}$}\\
$D_{4,2}$ & $\langle 16,7 \rangle$, $\langle 8,3 \rangle$, $\langle 8,1 \rangle$ & $M\Bigl(\sqrt{\sqrt{-7}+7}\Bigr)$ & $K\Bigl(\sqrt{-\sqrt{-2}\sqrt{\sqrt{-7}-7}+2\sqrt{-7}}\Bigr)$\vspace{1pt}\\
\multicolumn{4}{l}{$\boldsymbol{y^2 =x^6 + 7x^5 + 10x^4 + 10x^3 + 15x^2 + 17x + 4}$}\\
$O_1$ & $\langle 48,29 \rangle$, $\langle 24,12 \rangle$, $\langle 24,3 \rangle$ & $M(u_7,u_8)$ & $K\Bigl(\sqrt{-u_8^3+u_8^2+5u_8+4}\Bigr)$\\\hline
\end{tabular}
\bigskip

$u_1^3-7u_1+7 = u_2^4+4 u_2^2+8 u_2+8= u_3^3-4u_3+4=u_4^4+22u_4+22= u_5^3+6u_5-8 = 0$,\vspace{4pt}\\
$u_6^3+5u_6-10 = u_7^3+5u_7+10 = u_9^4+4u_9^2+8u_9+2=0$,\vspace{4pt}\\
$v_1^{12} - 12v_1^{11} + 70v_1^{10} - 236v_1^9 + 337v_1^8 - 40v_1^7 - 420v_1^6 + 452v_1^5 - 150v_1^4 + 16v_1^3 - 28v_1^2 + 8v_1 + 1=0$,\vspace{4pt}\\
$v_2^{12} + 44v_2^{11} + 682v_2^{10} + 4048v_2^9 + 3135v_2^8 - 19844v_2^7 + 306614v_2^6 + 1783540v_2^5 -$\\$ 5571929v_2^4 + 85184v_2^3 + 1269774v_2^2 - 1293732v_2 - 970299  = 0$,\vspace{4pt}\\
\end{center}
\end{table}

\begin{table}
\begin{center}
\caption{Twists of $C^0_3\colon y^2=x^6+1$ realizing each triple.}\label{Table: curves3}
\vspace{8pt}
\scriptsize
\setlength{\extrarowheight}{1pt}
\begin{tabular}{llll}
$G$ & $[\Gal(L/\Q), \Gal(K/\Q), \Gal(L/M)]$ &$K$ &$L$\\\hline
\multicolumn{4}{l}{$\boldsymbol{y^2 = x^6 + 6x^5 + 30x^4 + 120x^2 - 96x + 64}$}\\
$J(C_2)$ & $\langle 8,2 \rangle$, $\langle 4,2 \rangle$, $\langle 4,1 \rangle$ & $M(\sqrt{5})$ & $K\Bigl(\sqrt{\sqrt{5}+5}\Bigr)$\vspace{1pt}\\
\multicolumn{4}{l}{$\boldsymbol{y^2 = x^5 + 10x^3 + 9x}$}\\
$J(C_2)$ & $\langle 8,3 \rangle$, $\langle 4,2 \rangle$, $\langle 4,2 \rangle$ & $M(i)$ & $K(\sqrt[4]{3})$\vspace{4pt}\\
\multicolumn{4}{l}{$\boldsymbol{y^2 = x^6 - 15x^4 - 20x^3 + 6x + 1}$}\\
$J(C_6)$ & $\langle 24,10 \rangle$, $\langle 12,5 \rangle$, $\langle 12,5 \rangle$ & $M(i,u_1)$ & $K(\sqrt[4]{3})$\vspace{4pt}\\
\multicolumn{4}{l}{$\boldsymbol{y^2 = x^5 + 20x^3 + 36x}$}\\
$J(D_2)$ & $\langle 16,11 \rangle$, $\langle 8,5 \rangle$, $\langle 8,3 \rangle$ & $M(i,\sqrt{2})$ & $K(\sqrt[4]{6})$\vspace{4pt}\\
\multicolumn{4}{l}{$\boldsymbol{y^2 = x^6 + 15x^4 + 20x^3 + 30x^2 + 18x + 5}$}\\
$J(D_3)$ & $\langle 24,5 \rangle$, $\langle 12,4 \rangle$, $\langle 12,1 \rangle$ & $M(\sqrt{5},u_1)$ & $K\Bigl(\sqrt{2\sqrt{5}+10}\Bigr)$\vspace{1pt}\\
\multicolumn{4}{l}{$\boldsymbol{y^2 = x^6 + 6x^5 + 40x^3 - 60x^2 + 72x - 32}$}\\
$J(D_3)$ & $\langle 24,6 \rangle$, $\langle 12,4 \rangle$, $\langle 12,4 \rangle$ & $M(i,u_2)$ & $K(\sqrt[4]{3})$\vspace{4pt}\\
\multicolumn{4}{l}{$\boldsymbol{y^2 = x^6 + 3x^5 + 10x^3 - 15x^2 + 15x - 6}$}\\
$J(D_6)$ & $\langle 48,38 \rangle$, $\langle 24,14 \rangle$, $\langle 24,8 \rangle$ & $M(i,\sqrt{2},u_3)$ & $K(\sqrt[4]{2})$\vspace{4pt}\\
\multicolumn{4}{l}{$\boldsymbol{y^2 = x^6 + 1}$}\\
$C_{2,1}$ & $\langle 2,1 \rangle$, $\langle 2,1 \rangle$, $\langle 1,1 \rangle$ & $M$ & $K$\vspace{4pt}\\
\multicolumn{4}{l}{$\boldsymbol{y^2 = x^6 + 15x^4 + 15x^2 + 1}$}\\
$C_{2,1}$ & $\langle 4,2 \rangle$, $\langle 2,1 \rangle$, $\langle 2,1 \rangle$ & $M$ & $K(\sqrt{2})$\vspace{4pt}\\
\multicolumn{4}{l}{$\boldsymbol{y^2 = -x^6 - 6x^5 + 30x^4 - 20x^3 - 15x^2 + 12x - 1}$}\\
$C_{6,1}$ & $\langle 6,2 \rangle$, $\langle 6,2 \rangle$, $\langle 3,1 \rangle$ & $M(u_1)$ & $K$\vspace{4pt}\\
\multicolumn{4}{l}{$\boldsymbol{y^2 = x^6 + 6x^5 - 30x^4 + 20x^3 + 15x^2 - 12x + 1}$}\\
$C_{6,1}$ & $\langle 12,5 \rangle$, $\langle 6,2 \rangle$, $\langle 6,2 \rangle$ & $M(u_1)$ & $K(i)$\vspace{4pt}\\
\multicolumn{4}{l}{$\boldsymbol{y^2 = x^6 - 1}$}\\
$D_{2,1}$ & $\langle 4,2 \rangle$, $\langle 4,2 \rangle$, $\langle 2,1 \rangle$ & $M(i)$ & $K$\vspace{4pt}\\
\multicolumn{4}{l}{$\boldsymbol{y^2 = 11x^6 + 30x^5 + 30x^4 + 40x^3 - 60x^2 + 120x - 88}$}\\
$D_{2,1}$ & $\langle 8,3 \rangle$, $\langle 4,2 \rangle$, $\langle 4,1 \rangle$ & $M(\sqrt{-2})$ & $K\Bigl(\sqrt{\sqrt{6}-2}\Bigr)$\vspace{1pt}\\
\multicolumn{4}{l}{$\boldsymbol{y^2 = x^6 - 15x^4 + 15x^2 - 1}$}\\
$D_{2,1}$ & $\langle 8,5 \rangle$, $\langle 4,2 \rangle$, $\langle 4,2 \rangle$ & $M(i)$ & $K(\sqrt{2})$\vspace{4pt}\\
\multicolumn{4}{l}{$\boldsymbol{y^2 = x^6 + 4}$}\\
$D_{3,2}$ & $\langle 6,1 \rangle$, $\langle 6,1 \rangle$, $\langle 3,1 \rangle$ & $M(\sqrt[3]{2})$ & $K$\vspace{4pt}\\
\multicolumn{4}{l}{$\boldsymbol{y^2 = x^6 + 12x^5 + 15x^4 + 40x^3 + 15x^2 + 12x + 1}$}\\
$D_{3,2}$ & $\langle 12,4 \rangle$, $\langle 6,1 \rangle$, $\langle 6,2 \rangle$ & $M(\sqrt[3]{3})$ & $K(\sqrt{-2})$\vspace{4pt}\\
\multicolumn{4}{l}{$\boldsymbol{y^2 = x^6 + 9x^5 - 60x^4 - 120x^3 + 240x^2 + 144x - 64}$}\\
$D_{6,1}$ & $\langle 12,4 \rangle$, $\langle 12,4 \rangle$, $\langle 6,1 \rangle$ & $M(i,u_4)$ & $K$\vspace{4pt}\\
\multicolumn{4}{l}{$\boldsymbol{y^2 = x^6 + 6x^5 - 30x^4 - 40x^3 + 60x^2 + 24x - 8}$}\\
$D_{6,1}$ & $\langle 24,8 \rangle$, $\langle 12,4 \rangle$, $\langle 12,1 \rangle$ & $M(\sqrt{-2},u_5)$ & $K\Bigl(\sqrt{\sqrt{6}-2}\Bigr)$\vspace{1pt}\\
\multicolumn{4}{l}{$\boldsymbol{y^2 = x^6 + 3x^5 + 15x^4 - 20x^3 + 60x^2 - 60x + 28}$}\\
$D_{6,1}$ & $\langle 24,14 \rangle$, $\langle 12,4 \rangle$, $\langle 12,4 \rangle$ & $M(\sqrt{-2},u_2)$ & $K(\sqrt{2})$\vspace{4pt}\\
\multicolumn{4}{l}{$\boldsymbol{y^2 = x^6 + 2}$}\\
$D_{6,2}$ & $\langle 12,4 \rangle$, $\langle 12,4 \rangle$, $\langle 6,2 \rangle$ & $M(\sqrt[6]{2})$ & $K$\vspace{4pt}\\
\multicolumn{4}{l}{$\boldsymbol{y^2 =x^6 + 6x^5 - 15x^4 + 20x^3 - 15x^2 + 6x - 1}$}\\
$D_{6,2}$ & $\langle 24,14 \rangle$, $\langle 12,4 \rangle$, $\langle 12,5 \rangle$ & $M(\sqrt{-2},u_6)$ & $K(i)$\vspace{4pt}\\\hline
\end{tabular}
\bigskip

$u_1^3-3u_1+1=u_2^3-3u_2+4=u_3^3+3u_3-2=u_4^3-15u_4-10=0$,\vspace{4pt}\\
$u_5^3-9u_5-6=u_7^3-6u_7-6 =0$.
\end{center}
\end{table}


\begin{thebibliography}{McK-Sta}

\bibitem[BK11]{BK}
G. Banaszak and K.S. Kedlaya, \emph{An algebraic Sato-Tate group and
Sato-Tate conjecture}, arXiv:1109.4449v1 (2011).

\bibitem[Magma]{Magma}
W. Bosma, J.J. Cannon, C. Fieker, and A. Steel (eds.), \emph{Handbook of {Magma} functions}, 2.16 edition, \url{http://magma.maths.usyd.edu.au/magma/handbook/}, 2010.

\bibitem[Car01]{Car01} G. Cardona, \emph{Models racionals de corbes de g\`enere
$2$}, Tesi Doctoral, Universitat Polit\`ecnica de Catalunya, 2001.

\bibitem[Car06]{Car06} G. Cardona, \emph{Representations of $G_k$-groups and twists of the
genus two curve $y^2=x^5-x$}, Journal of Algebra \textbf{303} (2006), 707--721.

\bibitem[CGLR99]{CGLR99} G. Cardona, J. Gonz\'alez, J-C. Lario, A. Rio, \emph{On curves of genus $2$ with Jacobian of $\GL_2-$type},
Manuscripta Math. \textbf{98} (1999), 37--54.

\bibitem[Fit10]{Fit10} F. Fit\'e, \emph{Artin representations attached to pairs of isogenous abelian varieties}, arXiv:1012.3390v1 (2010).

\bibitem[FKRS12]{FKRS} F. Fit\'e, K.S. Kedlaya, V. Rotger, A.V. Sutherland, \emph{Sato-Tate distributions and Galois endomorphism modules in genus $2$}, to appear in Compositio Mathematica, arXiv:1110.6638v2 (2012).

\bibitem[FL11]{FL11} F. Fit\'e, J-C. Lario, \emph{The twisting representation of the $L$-function of a curve}, arXiv:1012.3393v1 (2011).

\bibitem[GAP]{GAP4}
  The GAP~Group, \emph{GAP -- Groups, Algorithms, and Programming}, version 4.4.12, \url{http://www.gap-system.org}, 2008.

\bibitem[Gr80]{Gr80}
B. H. Gross, \emph{Arithmetic on Elliptic Curves with Complex Multiplication},  Lecture Notes in Mathematics \textbf{776}, Springer, 1980.

\bibitem[He20]{He20} E. Hecke, \emph{Eine neue Art von Zetafunktionen und ihre Beziehungen zur Verteilung der Primzahlen}. Zweite Mitteilung, Math. Zeit.~\textbf{6} (1920), 11--51.

\bibitem[SGL]{SGL} H.U. Besche, B. Eick, E. O'Brien, \emph{Small Groups Library}, \emph{A millennium project: constructing Small Groups}, International Journal of Algebra and Computation {\bf 12} (2001), 623--644.

\bibitem[KS08]{KS08} K.S. Kedlaya, A.V. Sutherland, \emph{Computing L-series of hyperelliptic curves},
Algorithmic Number Theory Symposium---ANTS VIII, Lecture Notes in Comp. Sci. \textbf{5011}, Springer, 2008, 312--326.

\bibitem[Igu60]{Igu} J. Igusa, \emph{Arithmetic variety of moduli for genus two}, Annals of Mathematics {\bf 72} (1960), 612--649.

\bibitem[MRS07]{MRS07} B. Mazur, K. Rubin, A. Silverberg, \emph{Twisting commutative algebraic groups}, Journal of Algebra \textbf{314} (2007), 419--438.

\bibitem[OEIS]{OEIS} OEIS, The On-Line Encyclopedia of Integer Sequences, published electronically at \url{http://oeis.org}, 2011.

\bibitem[Pila90]{Pila90} J. Pila, \emph{Frobenius maps of abelian varieties and finding roots of unity in finite fields},
Mathematics of Computation \textbf{55} (1990), 745--763.

\bibitem[RS11]{RS} K. Rubin, A. Silverberg, \emph{Choosing the correct elliptic curve in the CM method}, Mathematics of Computation, {\bf 79} (2010), 545--561.

\bibitem[Se68]{Se68} J.-P. Serre, {\em Abelian $\ell$-adic Representations and Elliptic Curves}, Research Notes in Mathematics \textbf{7}, A K Peters, 1998.

\bibitem[Se12]{Se12} J.-P. Serre, \emph{Lectures on $N_X(p)$}, Research Notes in Mathematics \textbf{11}, CRC Press, 2012.

\bibitem[Sil09]{Sil09} J. H. Silverman, \emph{The Arithmetic of Elliptic Curves}, 2nd edition, Springer, New York, 2009.

\bibitem[Su07]{S07} A.V. Sutherland, \emph{Order computations in generic groups}, PhD thesis, Massachusetts Institute of Technology, 2007.

\bibitem[Su11a]{S11a} A.V. Sutherland, \emph{Structure computation and discrete logarithms in finite abelian $p$-groups}, Mathematics of Computation {\bf 80} (2011), 477--500.

\bibitem[Su11b]{S11b} A.V. Sutherland, \texttt{smalljac} software library, version 4.0, available at \url{http://math.mit.edu/~drew}, 2011.

\end{thebibliography}
\end{document}